\documentclass[11pt,a4paper,reqno]{amsart}

\usepackage{amsmath,amsfonts,amssymb,amsthm}
\usepackage{mathtools}
\usepackage{epsfig}
\usepackage{graphicx}
\usepackage{url}
\usepackage[usenames,dvipsnames]{xcolor}
\usepackage[colorlinks=true,linkcolor=Blue,citecolor=Green]{hyperref}
\usepackage{nicefrac}
\usepackage{aliascnt}       
\usepackage[alwaysadjust]{paralist}
\usepackage[T1]{fontenc}
\usepackage[utf8]{inputenc}
\usepackage{tikz,xifthen}
\usepackage{mytikz}
\usepackage{accents}

\usepackage{caption}
\usepackage[labelfont=rm]{subcaption}

\usepackage{blkarray}

\usetikzlibrary{matrix}

\usepackage[textwidth=2.5cm,textsize=small]{todonotes}

\newtheorem{theorem}{Theorem}[section]
\newaliascnt{lemma}{theorem}
\newaliascnt{corollary}{theorem}
\newaliascnt{definition}{theorem}
\newaliascnt{remark}{theorem}
\newaliascnt{proposition}{theorem}
\newaliascnt{conjecture}{theorem}
\newaliascnt{example}{theorem}
\newaliascnt{problem}{theorem}
\newaliascnt{question}{theorem}
\newaliascnt{claim}{theorem}

\newtheorem{lemma}[lemma]{Lemma}
\aliascntresetthe{lemma}

\newtheorem*{lemma*}{Lemma}

\newtheorem{corollary}[corollary]{Corollary}
\aliascntresetthe{corollary}

\newtheorem*{corollary*}{Corollary}

\newtheorem{definition}[definition]{Definition}
\aliascntresetthe{definition}

\newtheorem*{definition*}{Definition}

\newtheorem{remark}[remark]{Remark}
\aliascntresetthe{remark}

\newtheorem*{remark*}{Remark}

\newtheorem{proposition}[proposition]{Proposition}
\aliascntresetthe{proposition}

\newtheorem*{proposition*}{Proposition}

\newtheorem{conjecture}[conjecture]{Conjecture}
\aliascntresetthe{conjecture}

\newtheorem*{conjecture*}{Conjecture}

\newtheorem{example}[example]{Example}
\aliascntresetthe{example}

\newtheorem*{example*}{Example}

\aliascntresetthe{problem}

\newtheorem*{problem*}{Problem}

\newtheorem{question}[question]{Question}
\aliascntresetthe{question}

\newtheorem*{question*}{Question}

\aliascntresetthe{claim}

\newtheorem*{claim*}{Claim}

\DeclareMathOperator{\conv}{conv}

\DeclareMathOperator{\diag}{diag}
\DeclareMathOperator{\eval}{eval}

\def\K{\mathcal{K}}
\def\KK{\mathbb{K}}

\def\P{\mathcal{P}}
\def\L{\mathcal{L}}
\def\F{\mathcal{F}}

\def\cR{\mathcal{R}}
\def\T{\mathcal{T}}
\def\TL{\mathfrak{L}}
\def\R{\mathbb{R}}

\def\Z{\mathbb{Z}}

\def\N{\mathbb{N}}
\def\Q{\mathbb{Q}}

\def\TT{\mathbb{T}}

\def\eps{\varepsilon}

\DeclareMathOperator{\Tr}{Tr}
\DeclareMathOperator{\vol}{vol}
\DeclareMathOperator{\tminor}{tm}
\DeclareMathOperator{\tvol}{tvol}
\DeclareMathOperator{\tlvol}{tbvol}
\DeclareMathOperator{\tlsurf}{tbsurf}

\DeclareMathOperator{\rvol}{rvol}
\DeclareMathOperator{\qtvol}{qtvol}
\DeclareMathOperator{\tdet}{tdet}

\DeclareMathOperator{\aff}{aff}
\DeclareMathOperator{\lin}{lin}

\DeclareMathOperator{\val}{val}
\DeclareMathOperator{\bt}{bt}
\DeclarePairedDelimiter{\card}{\lvert}{\rvert}
\newcommand{\one}{\mathbf{1}}
\newcommand{\zero}{\mathbf{0}}
\DeclareMathOperator{\LogOp}{Log}
\newcommand{\Log}[1]{\LogOp\,\card{#1}}
\newcommand{\shp}[1]{\#\!\left(#1\right)}
\newcommand{\acirc}[1]{\accentset{\circ}{#1}}

\newcommand{\bLd}{\Gamma_b^d}

\DeclareMathOperator{\sgn}{sgn}
\DeclareMathOperator{\tconv}{tconv}
\DeclareMathOperator{\trk}{trk}

\newcommand\pseries[2]{#1\{\!\!\{#2\}\!\!\}} 

\DeclareFontFamily{U}{mathx}{\hyphenchar\font45}
\DeclareFontShape{U}{mathx}{m}{n}{
      <5> <6> <7> <8> <9> <10>
      <10.95> <12> <14.4> <17.28> <20.74> <24.88>
      mathx10
      }{}
\DeclareSymbolFont{mathx}{U}{mathx}{m}{n}
\DeclareFontSubstitution{U}{mathx}{m}{n}
\DeclareMathAccent{\widebar}{0}{mathx}{"73}

\makeatletter
\def\input@path{{./Images/}{./}}
\makeatother

\begin{document}

\title[Tropical Ehrhart Theory and Tropical Volume]{Tropical Ehrhart Theory\\ and Tropical Volume}

\author{Georg Loho}
\address{London School of Economics and Political Science\\
  Houghton Street\\
  London\\
  WC2A 2AE\\
UK}
\email{g.loho@lse.ac.uk}

\author{Matthias Schymura}
\address{Institute for Mathematics\\
         \'{E}cole Polytechnique F\'{e}d\'{e}rale de Lausanne\\
         CH-1015 Lausanne\\
         Switzerland}
\email{matthias.schymura@epfl.ch}

\thanks{The authors were supported by the Swiss National Science Foundation (SNSF) within the project \emph{Convexity, geometry of numbers, and the complexity of integer programming (Nr.~163071)}. GL was additionally supported by the ERC Starting Grant ScaleOpt, and MS by the SNSF project \emph{Lattice Algorithms and Integer Programming (Nr.~185030)}.}

\subjclass[2010]{14T05, 52C07, 52A38} 



\begin{abstract}
  We introduce a novel intrinsic volume concept in tropical geometry.
  This is achieved by developing the foundations of a tropical analog of lattice point counting in polytopes. 
  We exhibit the basic properties and compare it to existing measures.
  Our exposition is complemented by a brief study of arising complexity questions. 
\end{abstract}

\maketitle

\section{Introduction}

Tropical geometry is the study of piecewise-linear objects defined over the $(\max,+)$-semiring that arises by replacing the classical addition `$+$' with `$\max$' and multiplication `$\cdot$' with `$+$'.
While this often focuses on combinatorial properties, see~\cite{ItenbergMikhalkinShustin:2009,Butkovic:2010}, we are mainly interested in metric properties. 
Measuring quantities from tropical geometry turned out to be fruitful for a better understanding of interior point methods for linear programming~\cite{AllamigeonBenchimolGaubertJoswig:2018} and principal component analysis of biological data~\cite{YoshidaZhangZhang:2019}.
Moreover, it has interesting connections with representation theory~\cite{JoswigSturmfelsYu:2007,zhang2018computing} and computational complexity~\cite{gaubertmaccaig2019approximating}.

Driven by this motivation, we develop a novel volume notion for tropical convex sets by a thorough investigation of the tropical analog of lattice point counting. 
This continues the investigation of intrinsic tropical metric properties that started around a tropical isodiametric inequality~\cite{DepersinGaubertJoswig:2017} and tropical Voronoi diagrams~\cite{criadojoswigsantos2019tropical}.

Tropical polytopes are finitely generated tropical convex sets, see~\eqref{eq:definition+tropical+polytope}. 
Former work only considered the lattice points $\Z^d$ in a $d$-dimensional polytope, see in particular~\cite{butkovicmaccaig2013oninteger}. 
This idea was used to measure its \emph{Euclidean volume} and deduce the hardness to compute it by counting the integer lattice points~\cite{gaubertmaccaig2019approximating}.
These lattice points arise naturally through the representation of affine buildings as tropical polytopes~\cite{JoswigSturmfelsYu:2007}.
However, we are more interested in lattice points which are conformal with the semiring structure.
Varying the semiring as explained in Section~\ref{sec:versions+convexity} leads to two natural notions: integer lattice points in polytopes over the $(\max, \cdot)$-semiring and their image under a logarithm map over the $(\max,+)$-semiring.
This is related to the concept arising from `dequantization' but we show in Section~\ref{sec:trop+vol+revisited} how our tropical volume concept differs from the existing ones~\cite{DepersinGaubertJoswig:2017}.





The main idea leading to our novel concept of tropical volume is the following:
For a classical polytope $P \subseteq \R^d$, the Euclidean volume describes the asymptotic behavior of its \emph{Ehrhart function} $L(P,k) = \shp{kP \cap \Z^d}$, that is, the function that counts lattice points from $\Z^d$ that are contained in the~$k^{th}$ dilate of the polytope~$P$.
This discretization further refines if~$P$ is a \emph{lattice polytope}, meaning that all its vertices belong to~$\Z^d$.
In fact, Ehrhart proved that in this case $L(P,k)$ agrees with a polynomial of degree at most~$d$, for every positive integral dilation factor $k \in \Z_{>0}$ (see~\cite[Ch.~3]{beckrobins2007computing}):
\[
L(P,k) = \sum_{i=0}^d c_d(P) k^i.
\]
The polynomial on the right hand side is known as the \emph{Ehrhart polynomial} of~$P$, and the crucial point for us is that
\[
\vol(P) = \lim_{k \to \infty} \frac{L(P,k)}{k^d} = c_d(P).
\]
Now, our approach towards an intrinsic tropical volume concept is to turn this discretization process around and to establish tropical analogs to the previously described classical ideas.
This will be done in four steps:
\begin{enumerate}[i)]

 \item We define a suitable concept of tropical lattice (depending on a fineness parameter) and tropical lattice polytopes in Section~\ref{subsect:tropical+lattices}.

 \item In Section~\ref{sec:tropEhrpolys}, we develop a tropical Ehrhart theory showing that the corresponding tropical Ehrhart function exhibits polynomial behavior.
 
 \item We then take the leading coefficient of the tropical Ehrhart polynomial as the definition of tropical volume.
   
 \item Finally, we extract the metric information that is independent of the fineness parameter of the tropical lattice by using its asymptotics, and extend it to all tropical polytopes, without any integrality restriction. This is implemented in Section~\ref{sec:trop+vol+lattice+points}.
\end{enumerate}
The development of our tropical Ehrhart theory rests on making the transition from $(\R,+,\cdot)$ to the tropical semiring $\TT = (\R \cup \{-\infty\},\max,+)$  in two steps.
More precisely, we first replace addition `+' by the maximum operation to obtain the semiring $S_{\max,\cdot} = (\R_{\geq 0},\max,\cdot)$, and second we observe that for any $b \in \N_{\geq2}$, the map $x \mapsto \log_b(x)$ induces a semiring isomorphism between $S_{\max,\cdot}$ and~$\TT$.

On the one hand, this point of view motivates to introduce tropical integers as $\log_b(\Z_{\geq0})$, leading to what we call the \emph{tropical $b$-lattice} $\log_b(\Z_{\geq0})^d$ with fineness parameter $b \in \N_{\geq2}$.
And on the other hand, it allows to transfer classical Ehrhart theory on complexes of lattice polytopes to an Ehrhart theory for lattice polytopes over the various semirings which we explicitly describe in Theorem~\ref{thm:Ehrhart-max-times}, Theorem~\ref{thm:Ehrhart-max-plus-times}, and Theorem~\ref{thm:tropEhrPoly}.
These results heavily rely on the interplay of the involved semirings associated with tropical geometry, cf.~\cite{Butkovic:2010}.
While this approach is very conceptual and offers a first understanding of tropical Ehrhart theory, it has the disadvantage of lacking a useful description of the coefficients of the resulting tropical Ehrhart polynomials.

Therefore, we take a second route based on the covector decomposition that allows to triangulate a tropical lattice polytope into so-called alcoved simplices which are both tropically and classically convex polytopes.
This leads to the explicit representations of tropical Ehrhart coefficients in Theorem~\ref{thm:tropEhrCoeffsRep}, and eventually to our desired intrinsic volume concept.
As the key insight here, we observe that counting tropical lattice points in tropical dilations of alcoved simplices amounts to counting usual lattice points in dilates of diagonally transformed alcoved simplices (Lemma~\ref{lem:EhrhartSimplex}).
To assemble the Ehrhart coefficients correctly from these pieces, we need a better understanding of lower-dimensional structures of the covector decomposition, which is achieved in Section~\ref{sect:alcovedtriang}. 

As the result of the four-step-process outlined above, we define the \emph{tropical barycentric volume} $\tlvol(P)$ of a tropical polytope $P \subseteq \TT^d$ as
\[
\tlvol(P) := \max_x \, (x_1+\ldots+x_d),
\]
where the maximum is taken over all points $x \in P$ that are contained in a $d$-dimensional cell of the polyhedral complex associated to~$P$.
Our choice of name will become clear later on.

In Section~\ref{sec:properties+tropical+volume}, we investigate basic properties of the tropical barycentric volume.
We prove that it satisfies the natural tropical analogs of the fundamental properties of the Euclidean volume: monotonicity, the valuation property, rotation invariance, homogeneity, non-singularity, and multiplicativity.
In this sense $\tlvol(\cdot)$ is a meaningful and intrinsic volumetric concept for tropical geometry.

Furthermore, in Section~\ref{sec:trop+vol+revisited} we compare the tropical barycentric volume with existing volumetric measures.
For instance, it turns out to be bounded by the tropical dequantized volume $\qtvol^+(\cdot)$ defined in~\cite{DepersinGaubertJoswig:2017}.
More precisely, if $P = \tconv(M)$ is the tropical polytope defined as the tropical convex hull of the columns of $M \in \TT^{d \times m}$, then we prove in Theorem~\ref{thm:tvol_qtvol} that
\begin{align}
\tlvol(P) \leq \qtvol^+(M).\label{eqn:tbvol+qtvol+intro}
\end{align}
Motivated by this inequality, we go a step further and work towards lower-dimensional volumetric measures in Section~\ref{sec:metric+estimates+lower+volumes}.
We propose natural generalizations of the tropical barycentric volume that may serve as adequate tropical versions of the classical intrinsic volumes (or quermassintegrals) (cf.~\cite{schneider1993convex}).
For example, we define a tropical lower barycentric $i$-volume $\tlvol_i^-(P)$ of~$P = \tconv(M)$ and prove that it is upper bounded by the maximal tropical determinant of an $(i \times i)$-submatrix of~$M$ (see Theorem~\ref{thm:tlvoli_tmi}).
This extends~\eqref{eqn:tbvol+qtvol+intro}, because $\qtvol^+(M)$ equals the maximal tropical determinant of a $(d \times d)$-submatrix of~$M$, by a result in~\cite{DepersinGaubertJoswig:2017}.

We close the paper with Section~\ref{sec:computational+aspects} in which we discuss computational aspects of the problem of computing the tropical barycentric volume.
We argue that the decision problem that asks whether the tropical barycentric volume of a given tropical polytope is non-vanishing is equivalent to checking feasibility of a tropical linear program, or to deciding winning positions in mean-payoff games.
Therefore, this decision problem lies in NP~$\cap$~coNP (cf.~\cite{gaubertmaccaig2019approximating}).
Based on the computation of the tropical barycenter of a tropical simplex, we moreover devise an algorithm to determine the tropical barycentric volume of a tropical $d$-polytope with $m$ vertices, that runs in time $O(\binom{m}{d+1} d^3)$.

\section{Tropical convexity and tropical lattices}

In this section, we fix the main notation of the paper, discuss the crucial concept of an \emph{$i$-trunk} of a tropical polytope, introduce the notion of tropical lattice leading to our tropical Ehrhart theory, and finally review the relationship between different versions of convexity relevant to our studies.

\subsection{Tropical polytopes and alcoved triangulations}
\label{sect:alcovedtriang}

We denote by $\TT = (\R \cup \{-\infty\}, \oplus, \odot)$ the \emph{max-tropical semiring}, where $\oplus$ denotes the $\max$ operation and $\odot$ denotes the classical addition `$+$'.
The \emph{tropical convex hull} of a set $V \subseteq \TT^d$ is defined by 
\begin{equation} \label{eq:definition+tropical+polytope}
\tconv(V) = \bigg\{\bigoplus_{j=1}^{n} \lambda_j \odot v_j : \lambda_1,\ldots,\lambda_n \in \TT, \bigoplus_{j=1}^n \lambda_j = 0, v_1,\ldots,v_n \in V \bigg\} .
\end{equation}
If $V$ is finite, this is called a \emph{tropical polytope}.
We will switch freely between matrices and the set of their columns. 
A set is \emph{tropically convex} if it contains the tropical convex hull of each of its finite subsets.
By the tropical Minkowski-Weyl theorem~\cite{GaubertKatz:2007}, there is a unique minimal set of points generating a tropical polytope; we call these points the \emph{vertices}.

The `type decomposition' due to Develin \& Sturmfels~\cite{DevelinSturmfels:2004} shows that each tropical polytope has a decomposition into \emph{polytropes}, which are classically and tropically convex polytopes~\cite{JoswigKulas:2010}.
Following~\cite{FinkRincon:2015}, we use the name \emph{covector decomposition} for this polyhedral complex formed by the polytropes.
The vertices of the covector decomposition are the \emph{pseudovertices} and the \emph{dimension} of the tropical polytope is the maximal dimension of a polytope in the complex.

For a tropical polytope $P \subseteq \TT^d$, let this family of open polytopes be denoted by~$\F_P$.
An element $T \in \F_P$ is called an \emph{$i$-tentacle element}, if it is not contained in the closure of any $(i+1)$-dimensional polytope $Q \in \F_P$.
The following subcomplexes of $\F_P$ will be important later on and thus deserve some initial studies.

\begin{definition}[$i$-trunk] \label{def:i+trunk}
Let $P$ be a tropical polytope and let $i \in \{1,\ldots,d\}$.
We define the \emph{$i$-trunk} of $P$ as
\[
\Tr_i(P) := \bigcup \left\{ F \in \F_P : \exists\, G \in \F_P\text{ with }\dim(G) \geq i\text{ such that }F \subseteq G \right\}.
\]
\end{definition}

\noindent This means, that we obtain $\Tr_i(P)$ from $P$ after removing every $(i-1)$-tentacle element.
We always have $P = \Tr_1(P) \supseteq \Tr_2(P) \supseteq \ldots \supseteq \Tr_d(P)$.
Note that a more general concept was introduced in~\cite[Def.~2.8]{bjoernerwachs1996shellableI} for arbitrary simplicial complexes, but it was not given a name there.
In their notation, we have $\Tr_i(P) = \F_P^{(i,d)}$.

The following example shows that the $2$-trunk of a $2$-dimensional tropical polytope in $4$-dimensional space is not necessarily connected. 

\begin{example} \label{ex:non+connected+trunk}
  The tropical polytope spanned by the following points
  \[
  \begin{blockarray}{cccccc}
    a & b & c & d & e & f \\
    \begin{block}{(cccccc)}
    0 & 1 & 0 & 9 & 9 & 9 \\
    0 & 0 & 1 & 9 & 9 & 9 \\
    9 & 9 & 9 & 0 & 1 & 0 \\
    9 & 9 & 9 & 0 & 0 & 1 \\
    \end{block}
  \end{blockarray}
  \]
  is visualized in Figure~\ref{fig:weird+tropical+line}.
  All pseudovertices are marked in purple, we have the additional pseudovertices
  \[
  \begin{blockarray}{ccc}
    p & q & r \\
    \begin{block}{(ccc)}
    9 & 1 & 9 \\
    9 & 1 & 9 \\
    1 & 9 & 9 \\
    1 & 9 & 9 \\
    \end{block}
  \end{blockarray}
   \enspace .
   \]
   The maximal cells of the corresponding covector decomposition, computed with \texttt{polymake}~\cite{gawrilowjoswig2000polymake}, are
   $
   \{rp,rq, ped, pfd, qba, qca\}.
   $
\end{example}

\begin{figure}[ht]
  \begin{tikzpicture}

  \coordinate (p) at (0,5,5);
  \coordinate (q) at (5,5,0);
  \coordinate (r) at (0,0,0);

  \coordinate (a) at (6,6,0);
  \coordinate (b) at (6,5,0);
  \coordinate (c) at (5,6,0);
  \coordinate (d) at (0,6,6);
  \coordinate (e) at (0,5,6);
  \coordinate (f) at (0,6,5);

  \draw[Linesegment] (r) -- (p);
  \draw[Linesegment] (r) -- (q);

  \draw[Linesegment, fill, opacity = 0.4] (p) -- (e) -- (d) -- cycle;
  \draw[Linesegment, fill, opacity = 0.4] (p) -- (f) -- (d) -- cycle;

  \draw[Linesegment, fill, opacity = 0.4] (q) -- (b) -- (a) -- cycle;
  \draw[Linesegment, fill, opacity = 0.4] (q) -- (c) -- (a) -- cycle;

  \draw[Linesegment] (p) -- (e) -- (d) -- cycle;
  \draw[Linesegment] (p) -- (f) -- (d) -- cycle;

  \draw[Linesegment] (q) -- (b) -- (a) -- cycle;
  \draw[Linesegment] (q) -- (c) -- (a) -- cycle;

  \node[Endpoint] at (p) [label=right:$p$]{};
  \node[Endpoint] at (q) [label=below:$q$]{};
  \node[Endpoint] at (r) [label=right:$r$]{};
  \node[Endpoint] at (a) [label=above:$a$]{};
  \node[Endpoint] at (b) [label=below:$b$]{};
  \node[Endpoint] at (c) [label=above:$c$]{};
  \node[Endpoint] at (d) [label=above:$d$]{};
  \node[Endpoint] at (e) [label=below:$e$]{};
  \node[Endpoint] at (f) [label=above:$f$]{};

\end{tikzpicture}
  \caption{A 4-dimensional tropical polytope whose 2-trunk is disconnected. }
  \label{fig:weird+tropical+line}
\end{figure}
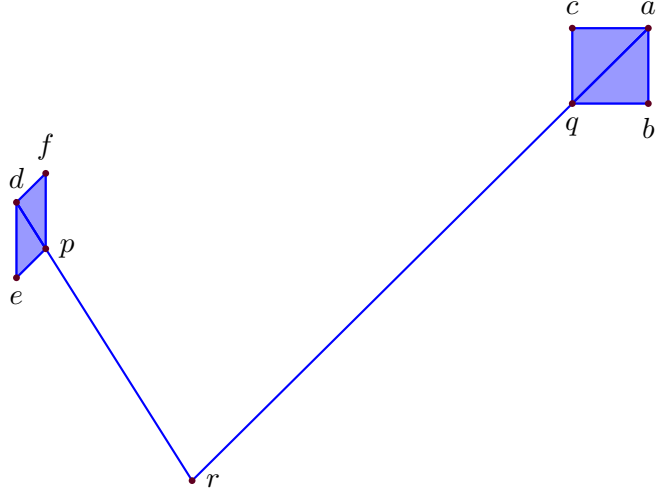

A particularly nice class of tropical polytopes are the \emph{pure tropical polytopes}, that is, those which coincide with their $d$-trunk.
The well-behaved nature of pure tropical polytopes was used to exhibit canonical exterior descriptions in~\cite{AllamigeonKatz:2017}.
In a similar spirit, the following statement uses a technique already occuring in the study of minimal external representations of tropical polytopes~\cite{AllamigeonKatz:2013}. 
In contrast to the disconnectedness of the $2$-trunk in Example~\ref{ex:non+connected+trunk}, it shows in particular that the $d$-trunk of a tropical polytope in~$\TT^d$ is a tropical polytope itself.

\begin{proposition} \label{prop:dtrunk+tropical+convex}
  The tropical convex hull of two full-dimensional pure tropical polytopes is a pure, full-dimensional tropical polytope.
  
  Consequently, the $d$-trunk of a tropical polytope in $\TT^d$ is a tropical polytope. 
\end{proposition}
\begin{proof}
  Let $P$ and $Q$ be two full-dimensional pure tropical polytopes in $\TT^d$ and let $\mathring{P}$ and $\mathring{Q}$ be their interior.
  Clearly, we have $\tconv(P \cup Q) \supseteq \overline{\tconv(\mathring{P} \cup \mathring{Q})}$, where $\overline{S}$ denotes the closure of a set~$S$.
  As $P$ and $Q$ are pure, we have $\overline{\mathring{P}} = P $ and $\overline{\mathring{Q}} = Q$.
  Let $t = \bigoplus_{r \in R} \lambda_r \odot r \oplus \bigoplus_{s \in S} \lambda_s \odot s$ for some finite subsets $R \subset P, S \subset Q$ be a point in $\tconv(P \cup Q)$ and let $(r_i)_{i \in \N} \to r$ for each $r \in R$ and $(s_i)_{i \in \N} \to s$ for each $s \in S$ be sequences in $\mathring{P}$ and $\mathring{Q}$, respectively.
  By the continuity of the operations $\max$ and `$+$', we obtain
  \[
  \left(\bigoplus_{r \in R} \lambda_r \odot r_i \oplus \bigoplus_{s \in S} \lambda_s \odot s_i\right)_{i \in \N} \to t \enspace .
  \]
  Together with the other inclusion, this shows $\tconv(P \cup Q) = \overline{\tconv(\mathring{P} \cup \mathring{Q})}$.
  For $\eps > 0$, we define
  \[
  B_{\eps} = \tconv
  \begin{pmatrix}
    -\eps & \eps & 0 & \ldots & 0 \\
    -\eps & 0 & \eps & \ldots & 0 \\
    \vdots & 0 & \ddots & \ddots & \vdots \\
    - \eps & 0 & \ldots & 0 & \eps 
  \end{pmatrix} \enspace ,
  \]
  a full-dimensional polytrope.

 For any two points $p \in \mathring{P}$ and $q \in \mathring{Q}$ there is a sufficiently small $\eps > 0$ such that $p + B_{\eps} \subseteq \mathring{P}$ and $q + B_{\eps} \subseteq \mathring{Q}$.
 Then the `inflated tropical line' $\tconv(p,q) + B_{\eps}$ is contained in $\tconv(\mathring{P} \cup \mathring{Q})$.
 Therefore, each point is surrounded by a small full-dimensional polytrope in $\tconv(\mathring{P} \cup \mathring{Q})$.
 This implies that each point of $\overline{\tconv(\mathring{P} \cup \mathring{Q})}$ is in the closure of a full-dimensional cell. 
  Hence, $\tconv(P \cup Q) = \overline{\tconv(\mathring{P} \cup \mathring{Q})}$ is pure and full-dimensional.

  The polytropes in the covector decomposition of the $d$-trunk are full-dimensional pure tropical polytopes $P$.
  Hence, the tropical convex hull of their union is a full-dimensional pure tropical polytope.
  Moreover, it is contained in the $d$-trunk of~$P$, as it is a subset of $P$.
  Therefore, the tropical convex hull of the $d$-trunk of $P$ is just the $d$-trunk itself. 
%
\end{proof}
   
The covector decomposition of a tropical polytope $P = \tconv(V)$, where~$V$ has only integral entries, is formed of alcoved polytopes in the sense of Lam \& Postnikov~\cite{LamPostnikov:2007}.
They studied triangulations and lattice points of alcoved polytopes from a classical point of view, while we are heading towards tropical metric estimates.
Each such alcoved polytope has a triangulation into simplices of the form
\[
  \Delta_\pi(a) := \conv\left\{ a + e_{\pi(1)} + \ldots + e_{\pi(\ell)} : \ell=0,1,\ldots,d \right\} \enspace ,
\]
where $\pi \in S_d$ is a permutation of the coordinates and $a \in \Z^d$.
For $\pi = id$, we just write $\Delta(a) := \Delta_{id}(a)$.
We denote the simplicial complex formed by these \emph{alcoved simplices} by $\T_P$ and call it the \emph{alcoved triangulation} of~$P$.
The inequality description of $\Delta(\zero)$ is given by
\[
\Delta(\zero) = \left\{ x \in \R^d : 0 \leq x_d \leq x_{d-1} \leq \ldots \leq x_1 \leq 1 \right\}
\]
(cf.~\cite[Ch.~7]{becksanyal2018combinatorial}), where the all-zeroes vector is denoted by $\zero = (0,\ldots,0)^\intercal$.
We use the following notation to compactly index (half-)open faces of $\Delta(\zero)$: For $s = (\prec_1,\prec_2,\ldots,\prec_{d+1}) \in \left\{=,\leq,<\right\}^{d+1}$, we write
\[
\Delta^s(\zero) = \left\{ x \in \R^d : 0 \prec_{d+1} x_d \prec_d x_{d-1} \prec_{d-1} \ldots \prec_2 x_1 \prec_1 1 \right\} \enspace ,
\]
and
\[
\Delta^s(a) = a + \Delta^s(\zero) \enspace .
\]

\subsection{Tropical lattices}
\label{subsect:tropical+lattices}

Recent advances on the complexity of linear programming using tropical geometry~\cite{AllamigeonBenchimolGaubertJoswig:2018} demonstrated a fruitful use of metric estimates for tropical polyhedra.
In classical convex geometry, the number of lattice points can be interpreted as a discrete version of a volume.
This raises the question what `tropical integers' or `tropical natural numbers' should be.

The non-negative integers form a submonoid of the additive monoid $(\R,+)$ generated by $1$.
The analogous tropical construction does not lead to a rich structure, as tropical addition is idempotent, and so $0 \oplus 0 = 0$.

Another approach comes from considering $\Z$ as the prime ring of characteristic $0$.
The orbit of $1 \in \R$ by multiplication with $\Z$ yields the integers~$\Z$.
The orbit of tropical one $0 \in \TT$ by tropical multiplication with $\Z$ again yields the set~$\Z$. 
Although this perspective has been used and allows a tropical Ehrhart theory (see Section~\ref{sec:tropEhrpolys}), it is too rough for our purposes.

Instead we propose to consider the set $\Gamma_b := \log_b(\Z_{\geq 0})$ as a concept for tropical integers, where $b \geq 2$ is an arbitrarily chosen natural number.
This is natural in the sense that it respects the operation-wise transition from $(\R,+,\cdot)$ to the tropical semiring $(\R \cup \{-\infty\},\max,+)$:
\[
(\Z,+,\cdot) \quad \longrightarrow \quad (\Z_{\geq 0},\max,\cdot) \quad \longrightarrow \quad (\log_b(\Z_{\geq 0}),\max,+).
\]
As additional motivation, we observe that~$\Gamma_b$ satisfies a tropicalization of the identity
\[
\shp{[0,k \cdot v) \cap \Z_{\geq 0}} = k \cdot \shp{[0,v) \cap \Z_{\geq 0}}, \quad\text{for}\quad k,v \in \Z_{\geq 0}.
\]
Indeed, we have
\[
\shp{[-\infty,k \odot v) \cap \Gamma_b} = b^k \cdot \shp{[-\infty,v) \cap \Gamma_b}, \; \text{for} \; k \in \Z_{\geq 0} \text{ and } v \in \Gamma_b \,.
\]

\noindent Our main concept of tropical lattice is therefore the following.

\begin{definition}[Tropical $b$-lattice]
Define the \emph{tropical $b$-lattice} in $\TT^d$ by
\[
\bLd := \left(\log_b(\Z_{\geq0})\right)^d = \left\{ \left(\log_b(x_1),\ldots,\log_b(x_d)\right)^\intercal : x_1,\ldots,x_d \in \Z_{\geq0} \right\}.
\]
\end{definition}

\noindent Tropical lattice polytopes are then defined accordingly.

\begin{definition}[Tropical lattice polytopes] \label{def:tropical+lattice+polytope}
Let $b \in \N_{\geq2}$.
A tropical polytope, whose vertices all lie in~$\bLd$, is called a \emph{tropical $b$-lattice polytope}.
As we often want to vary $b$, we define \emph{(canonical) tropical lattice polytopes} as those whose vertices lie in $\TT\N^d := (\Z_{\geq0} \cup \{-\infty\})^d$. 
\end{definition}

Canonical tropical lattice polytopes were already studied with a different motivation by Zhang~\cite{zhang2018computing}.
Observe that the vertices of a canonical tropical lattice polytope lie in all tropical $b$-lattices, that is,
\[
\TT\N^d \subseteq \bigcap_{b \in \N_{\geq 2}} \log_b(\Z_{\geq0})^d \enspace .
\]
In fact, for every $m \in \TT\N = \Z_{\geq0} \cup \{-\infty\}$ and every $b \in \N_{\geq2}$ we have $m = \log_b(b^m)$, with the convention that $b^{-\infty}=0$.

\subsection{Different versions of convexity} \label{sec:versions+convexity}

Tropical convexity is mainly associated with the semiring $S_{\max,+} = (\R \cup \{-\infty\}, \max, +)$ or, by applying the semiring isomorphism $x \mapsto -x$, the semiring $(\R \cup \{\infty\}, \min, +)$.
In the notation introduced before, we have $\TT = S_{\max,+}$.
We use the latter notation whenever we need to emphasize the different semirings, and we employ the shorter and more common notation~$\TT$ otherwise.

While transferring from $S_{\max,+}$ to the semiring $S_{\max,\cdot} = (\R_{> 0} \cup \{0\},\max,\cdot)$ via the semiring isomorphism $\exp_b \colon x \mapsto b^x$ is often merely a structural reformulation, it has a benefit for our metric considerations.
The next claim is far from true for general polytopes but due to the special structure of polytropes.

\begin{proposition} \label{prop:exp+img+polytrope}
  The image under the map $\exp_b$ of a polytrope is a polytope.
\end{proposition}
\begin{proof}
  The defining inequalities of a polytrope are of the form $c \leq x_i$, $c \geq x_i$, or $x_i \leq x_j + c$, for $i \neq j$, see~\cite{JoswigKulas:2010}. As
  \[
  \exp_b\left(\{x \in \R^d \colon c \geq x_i \}\right) = \{x \in \R^d \colon \exp_b(c) \geq x_i \} \enspace ,
  \]
  and analogously with $\leq$ instead of $\geq$, as well as, 
  \[
  \exp_b\left(\{x \in \R^d \colon x_i \leq x_j + c \}\right) = \{x \in \R^d \colon x_i \leq \exp_b(c) \cdot x_j \} \enspace ,
  \]
  the statement follows by taking the intersection of such sets. 
\end{proof}

Note that, more generally, the image of a \emph{weighted digraph polyhedron}~\cite{JoswigLoho:2016} under the exponentiation map results in a particular \emph{distributive polyhedron} as studied in~\cite{FelsnerKnauer:2011}.

While a \emph{polytope over $S_{\max,+}$} is a just a tropical polytope as defined in~\eqref{eq:definition+tropical+polytope}, its image under a semiring isomorphism $\exp_b$, for some $b \in \R_{\geq 0}$, is a \emph{polytope over $S_{\max,\cdot}$}.
Proposition~\ref{prop:exp+img+polytrope} shows that we obtain a polyhedral complex subdividing a polytope~$P$ over $S_{\max,\cdot}$, as the image of the covector decomposition of the polytope $\log_b(P)$ over $S_{\max,+}$.
We call this again the covector decomposition and its vertices the pseudovertices. 

Furthermore, we consider the field of generalized convergent Puiseux series~\cite{DriesSpeissegger:1998,AllamigeonBenchimolGaubertJoswig:2018}, which we denote by $\KK = \pseries{\R}{t}$.
This gives rise to the semiring $\KK_{\geq 0} := (\KK_{> 0} \cup \{0\},+,\cdot)$.
There is a semiring homomorphism $\val$ to $S_{\max,+}$ mapping a series to its lowest exponent and mapping $0$ to $-\infty$.
By construction, we can also use the evaluation map $\eval_r$ to evaluate an element of $\KK$ at some suitable $r \in \R$.
By~\cite{DevelinYu:2007}, each polytope over $S_{\max,+}$ is the image under $\val$ of some polytope over $\KK_{\geq 0}$.
Additionally, each polytope over $\KK_{\geq 0}$ yields a polytope over $\R$ by evaluation at some suitable $r \in \R$.
Note that we don't use the evaluation map in the following but we want to emphasize on the strong (metric) connection between polytopes over $S_{\max,+}$ and over $\R$. 

A summary of the semiring isomorphisms and other involved maps is shown in Figure~\ref{fig:diagram+relation+convexities}.

\begin{figure}[ht]
  \begin{tikzpicture}

  \matrix (m) [matrix of math nodes,row sep=3em,column sep=5em,minimum width=2em]
          {
            (\R, +, \cdot) & (\KK_{> 0} \cup \{0\}, +, \cdot) \\
            (\R_{> 0} \cup \{0\}, \max, \cdot) & (\R\cup\{-\infty\}, \max, +) \\
            (\R_{> 0} \cup \{\infty\}, \min, \cdot) & (\R\cup\{\infty\}, \min, +) \\};
          \path[]
          (m-1-1) edge [<-] node [below] {$\eval_r$} (m-1-2)
          (m-1-2) edge [->] node [right] {$\val$} (m-2-2)
          (m-2-1) edge [<->] node [left] {$\frac{1}{x}$} (m-3-1)
          ([yshift=0.4ex]m-2-1.east) edge [->] node [above] {$\log_b$}([yshift=0.4ex]m-2-2.west)
          ([yshift=-0.4ex]m-2-1.east) edge [<-] node [below] {$\exp_b$} ([yshift=-0.4ex]m-2-2.west)
          ([yshift=0.4ex]m-3-1.east|-m-3-2) edge [->] node [above] {$\log_b$} ([yshift=0.4ex]m-3-2.west)
          ([yshift=-0.4ex]m-3-1.east|-m-3-2) edge [<-] node [below] {$\exp_b$} ([yshift=-0.4ex]m-3-2.west)
          (m-2-2) edge [<->] node [right] {$-x$} (m-3-2);
\end{tikzpicture}
  \caption{Commutative diagram of several semiring isomorphisms and the connection with Puiseux series. }
  \label{fig:diagram+relation+convexities}
\end{figure}
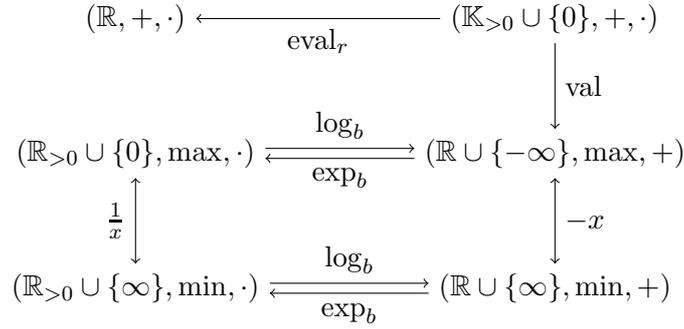

\section{Tropical Ehrhart Polynomials}
\label{sec:tropEhrpolys}

\subsection{Lattice point counting and semiring isomorphisms}

Ehrhart's theorem on the polynomiality of the counting function $k \mapsto L(P,k)$ of a lattice polytope $P \subseteq \R^d$, has the following powerful extension to complexes of lattice polytopes.

\begin{theorem}[{\cite[Cor.~5.6.1]{becksanyal2018combinatorial}}]
\label{thm:Ehrhart-latpoly-complex}
Let $\K$ be a complex of lattice polytopes in~$\R^d$ and let $\card{\K} = \bigcup_{P \in \K} P$ be its underlying point set.
Then, the counting function $k \mapsto \shp{k\card{\K} \cap \Z^d}$ agrees with a polynomial of degree $\dim(\K)$ for all positive integers $k \in \N$.
\end{theorem}

We saw in Section~\ref{sec:versions+convexity} that a polytope over $S_{\max,\cdot}$ has a natural structure as a polyhedral complex.
We call such a polytope a \emph{lattice polytope} if all pseudovertices in the covector decomposition are lattice points.
Thus, Theorem~\ref{thm:Ehrhart-latpoly-complex} yields the following.

\begin{theorem} \label{thm:Ehrhart-max-times}
  For a lattice polytope $P \subseteq (S_{\max,\cdot})^d$, the counting function $k \mapsto \shp{kP \cap \Z^d}$ agrees with a polynomial of degree $\dim(P) \leq d$ for all positive integers $k \in \N$.

The coefficient in front of $k^d$ equals the Euclidean volume of $P$. 
\end{theorem}

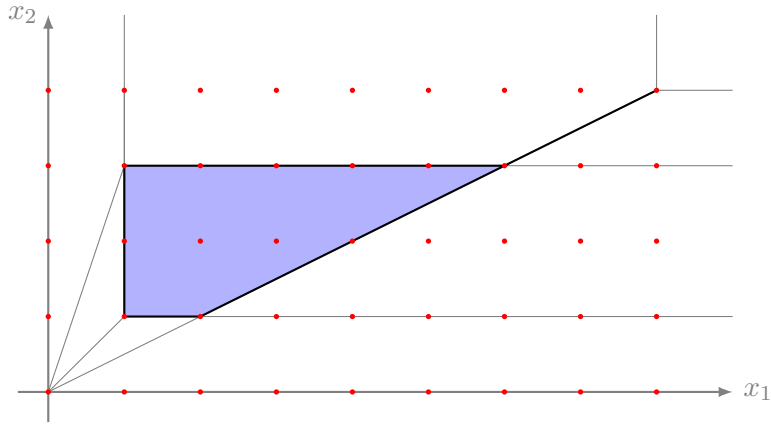
\begin{figure}[ht]
  \begin{tikzpicture}

  \Koordinatenkreuz{-0.4}{9}{-0.4}{5}{$x_1$}{$x_2$}

  \coordinate (a) at (1,1);
  \coordinate (b) at (1,3);
  \coordinate (c) at (8,4);

  \coordinate (origin) at (0,0);

  \draw[Sectorsegment] (origin) -- (a);
  \draw[Sectorsegment] (a) -- +(0,3);
  \draw[Sectorsegment] (a) -- +(8,0);
  \draw[Sectorsegment] (origin) -- (b);
  \draw[Sectorsegment] (b) -- +(0,2);
  \draw[Sectorsegment] (b) -- +(8,0);
  \draw[Sectorsegment] (origin) -- (c);
  \draw[Sectorsegment] (c) -- +(0,1);
  \draw[Sectorsegment] (c) -- +(1,0);

  \draw[black,thick, fill=blue!30] (6,3) -- (2,1) -- (1,1) -- (1,3) -- cycle -- (8,4);

  \foreach \x in {0,...,8}
  \foreach \y in {0,...,4}
  \node[LatticePoint] at (\x,\y) {};  
  
\end{tikzpicture}
  \caption{A lattice polytope over $S_{\max,\cdot}$ }
  \label{fig:trop+lat+poly+times+1}
\end{figure}

\noindent A natural question that arises is

\begin{question}
How can we characterize the vertices of those polytopes over $S_{\max,\cdot}$ for which all the pseudovertices are lattice points?
\end{question}

Going back to canonical tropical lattice polytopes (see Definition~\ref{def:tropical+lattice+polytope}), we actually obtain two different polynomials; one counting the lattice points in~$\Z^d$, the other one counting $b$-lattice points.
The first version is less natural from the semiring operations, but it was used in~\cite{gaubertmaccaig2019approximating}.

\begin{theorem} \label{thm:Ehrhart-max-plus-times}
For a canonical tropical lattice polytope $P \subseteq (S_{\max,+})^d = \TT^d$ the counting function $k \mapsto \shp{k\cdot P \cap \Z^d}$ agrees with a polynomial of degree $\dim(P)$ for all positive integers $k \in \N$.  
\end{theorem}

\begin{figure}[ht]
  \begin{subfigure}[h]{0.49\textwidth}
  \begin{tikzpicture}

  \Koordinatenkreuz{-0.7}{4.7}{-0.7}{3.7}{$x_1$}{$x_2$}

  \coordinate (a) at (0,0);
  \coordinate (b) at (0,{log2(3)});
  \coordinate (c) at (3,2);

  \coordinate (origin) at (0,0);

  \draw[Sectorsegment] (a) -- +(-0.6,-0.6);
  \draw[Sectorsegment] (a) -- +(0,2);
  \draw[Sectorsegment] (a) -- +(3,0);
  \draw[Sectorsegment] (b) -- +(-0.6,-0.6);
  \draw[Sectorsegment] (b) -- +(0,1);
  \draw[Sectorsegment] (b) -- +(3,0);
  \draw[Sectorsegment] (c) -- +(-2.6,-2.6);
  \draw[Sectorsegment] (c) -- +(0,0.4);
  \draw[Sectorsegment] (c) -- +(0.4,0);

  \draw[black,thick, fill=blue!30] ({log2(6)},{log2(3)}) -- (b) -- (a) -- (1,0) -- cycle -- (c); 

  \foreach \x in {1,...,16}
  \foreach \y in {1,...,8}
  \node[LatticePoint] at ({log2(\x)},{log2(\y)}) {};  
  
\end{tikzpicture}
  \caption{The $2$-lattice polytope over $S_{\max,+}$ arising as the $\log_2$-image of Figure~\ref{fig:trop+lat+poly+times+1}. }
  \label{subfig:trop+lat+poly+plus+1}
  \end{subfigure}
  \begin{subfigure}[h]{0.49\textwidth}
  \begin{tikzpicture}
  \Koordinatenkreuz{-0.7}{4.7}{-0.7}{3.7}{$x_1$}{$x_2$}

  \coordinate (a) at (1,1);
  \coordinate (b) at (1,{log2(6)});
  \coordinate (c) at (4,3);

  \coordinate (origin) at (0,0);


  \draw[black,thick, fill=blue!30] ({log2(12)},{log2(6)}) -- (b) -- (a) -- (2,1) -- cycle -- (c);

  \foreach \x in {1,...,16}
  \foreach \y in {1,...,8}
  \node[LatticePoint] at ({log2(\x)},{log2(\y)}) {};  
  
\end{tikzpicture}
  \caption{The same $2$-lattice polytope (tropically) dilated by $1$. }
  \end{subfigure}
  \caption{The lattice points are condensed by $\log_2$ such that the tropical dilation implies a polynomial increase in the number of contained lattice points. }
\end{figure}
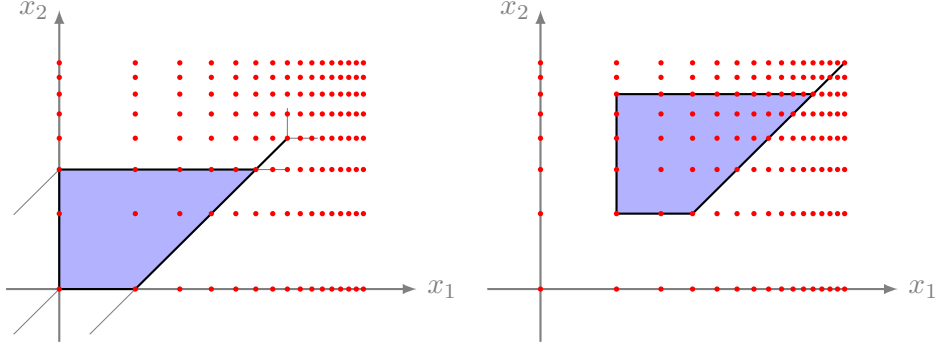

The next concept is at the heart of our quantitative studies. 

\begin{definition}
Let $P \subseteq \TT^d$ be a tropical lattice polytope and let $b \in \N_{\geq2}$.
We define the \emph{tropical lattice point enumerator} of $P$ (with respect to~$b$) as
\[
\TL_P^b(k) := \shp{ (k \odot P) \cap \bLd } , \quad k \in \Z_{\geq0}.
\]
\end{definition}

\noindent Applying the semiring isomorphism $\log_b$ to Theorem~\ref{thm:Ehrhart-max-times} we obtain

\begin{theorem}[Tropical Ehrhart polynomial]
\label{thm:tropEhrPoly}
Let $b \in \N_{\geq 2}$ and let $P \subseteq \TT^d$ be a tropical lattice polytope.
The tropical lattice point enumerator $\TL_P^b(k)$ agrees with a polynomial in $b^k$ for every $k \in \Z_{\geq0}$. 
\end{theorem}

\begin{proof}
  The set $Q = \exp_b(P)$ is a lattice polytope over $S_{\max,\cdot}$. 
  Hence, by Theorem~\ref{thm:Ehrhart-max-times}, there is a polynomial $q$ of degree $\dim(Q) = \dim(P)$ with
  \[
  q(\ell) = \shp{\ell Q \cap \Z^d}
  \]
  for all $\ell \in \Z_{>0}$.
  Substituting $b^k$ for $\ell$ and using $Q = \exp_b(P) \subseteq \R_{\geq 0}^d$ we get
  \[
  q(b^k) = \shp{\log_b\left(b^k \cdot \exp_b(P) \cap \Z_{\geq 0}^d \right)} = \shp{(k \odot P) \cap \bLd} .
  \]
  Note the use of the semiring homomorphism property of $\log_b$. 
\end{proof}

\begin{remark}
The proof above shows that the Ehrhart polynomials of $P=\tconv(M) \subseteq (S_{\max,+})^d$ and $Q=\exp_b(P) \subseteq (S_{\max,\cdot})^d$ agree up to a change of variables.
More precisely, we have $\TL_P^b(k) = q(b^k)$, for all $k \in \Z_{\geq 0}$.
\end{remark}

\begin{remark}
\label{rem:quasi-polynomials}
If one relaxes the integrality assumption in the classical setting and considers \emph{rational} polytopes $P \subseteq \R^d$, that is, polytopes all of whose vertices have only rational coordinates, then their Ehrhart function $k \mapsto \shp{kP \cap \Z^d}$ turns out to be a quasi-polynomial (cf.~\cite[Ch.~3.8]{beckrobins2007computing}).

In the various scenarios discussed above, rationality is defined as follows:
\begin{itemize}
 \item a polytope over $S_{\max,\cdot}$ is \emph{rational} if all its pseudovertices are rational,
 
 \item a polytope over $S_{\max,+}$ is \emph{tropically rational} if all its pseudovertices are integral (allowing possibly negative coordinates),
 
 \item a tropical polytope $P \subseteq \TT^d$ is \emph{tropically $b$-rational} if all its vertices are contained in $\log_b(\Q_{\geq 0})^d$.
\end{itemize}
The methods that we employed above to prove polynomiality, can similarly be used to show that in all three cases above the corresponding Ehrhart functions are quasi-polynomials as well.
\end{remark}

\begin{definition}[Tropical Ehrhart coefficients]
Let $P \subseteq \TT^d$ be a tropical lattice polytope.
We write
\[
\TL_P^b(k) = \sum_{i=0}^d c_i^b(P) (b^k)^i
\]
for its tropical Ehrhart polynomial and we call $c_i^b(P)$ the \emph{$i^{th}$ tropical Ehrhart coefficient} of~$P$.
\end{definition}

A very useful and reoccurring phenomenon in geometric combinatorics is \emph{reciprocity} (see~\cite{becksanyal2018combinatorial} for a detailed account).
For lattice point counting functions this is known as Ehrhart--MacDonald reciprocity (cf.~\cite[Ch.~4]{beckrobins2007computing}), and refers to the fact that evaluating the Ehrhart polynomial $L(P,k) = \sum_{i=0}^d c_i(P) k^i$ of a lattice polytope $P \subseteq \R^d$ at negative integers amounts to counting lattice points in the $k^{th}$ dilate of the interior $\acirc{P}$ of~$P$.
That is,
\begin{align*}
L(P,-k) &= (-1)^d L(\acirc{P},k),\quad\text{for}\quad k \in \Z_{>0}.
\end{align*}
We say that a counting function satisfying this relation \emph{fulfills reciprocity}.

If a lattice polytope over $S_{\max,\cdot}$ is pure, defined analogously for polytopes over $S_{\max,\cdot}$ as over $S_{\max,+}$, the polyhedral complex induced by its covector decomposition is a $d$-manifold and by~\cite{Macdonald:1971} reciprocity holds. 

\begin{theorem}\
\begin{enumerate}[i)]
 \item The Ehrhart polynomial in Theorem~\ref{thm:Ehrhart-max-times} of a pure lattice polytope over $S_{\max,\cdot}$ and the Ehrhart polynomial in Theorem~\ref{thm:Ehrhart-max-plus-times} of a pure canonical tropical lattice polytope over $S_{\max,+}$ fulfill reciprocity.
 
 \item The tropical Ehrhart polynomial $\TL_P^b(k)$ of a pure tropical lattice polytope $P \subseteq \TT^d$ fulfills reciprocity, in the sense that
 \[
 c_i^b(\acirc{P}) = (-1)^{d-i} c_i^b(P),\quad\text{for}\quad i \in \{0,1,\ldots,d\}.
 \]
\end{enumerate}
\end{theorem}



\subsection{Explicit expressions for tropical Ehrhart coefficients}

In this section, we take a much more refined route to Theorem~\ref{thm:tropEhrPoly} which is based on combining the covector decomposition with tools from classical Ehrhart theory.
This allows for a refined representation of the tropical Ehrhart coefficients, and leads to our desired tropical volume concept.
For comparison to ordinary Ehrhart theory and further reading, we refer to~\cite{beckrobins2007computing}. 

In order to formulate our main technical lemma, we denote the diagonal matrix with diagonal entries $b^{a_i}$ by $D_b^a = \diag(b^{a_1},\ldots,b^{a_d}) \in \Z^{d \times d}$, for $a \in \Z_{\geq0}^d$ and $b \in \N_{\geq 2}$.
Further, let $s = (\prec_1,\prec_2,\ldots,\prec_{d+1}) \in \left\{=,\leq,<\right\}^{d+1}$.
We denote the all-one vector by $\one = (1,\ldots,1)^\intercal$.

\begin{lemma}
\label{lem:EhrhartSimplex}
For every $k \in \Z_{\geq0}$, the map $\phi : \R_{>0}^d \to \R^d$ defined by $\phi(z) = (\log_b(z_1),\dots,\log_b(z_d))^\intercal$ induces a bijection between
\[
\left( b^k D_b^a \one + (b^{k+1}-b^k) D_b^a \Delta^s(\zero) \right) \cap \Z_{\geq0}^d \quad\text{and}\quad \left( k \odot \Delta^s(a) \right) \cap \bLd .
\]
\end{lemma}

\begin{proof}
Clearly, $\phi$ is bijective and by definition it maps points in $\Z_{\geq0}^d$ to points in $\bLd = (\log_b(\Z_{\geq0}))^d$.
So what we need to check is that $z \in b^k D_b^a \one + (b^{k+1}-b^k) D_b^a \Delta^s(\zero)$ if and only if $\phi(z) \in \left( k \odot \Delta^s(a) \right) = k \one + a + \Delta^s(\zero)$.
As we saw above, the inequality description of the simplex $\Delta^s(\zero)$ is given by
\[
\Delta^s(\zero) = \left\{ x \in \R^d : 0 \prec_{d+1} x_d \prec_d x_{d-1} \prec_{d-1} \ldots \prec_2 x_1 \prec_1 1 \right\}.
\]
Therefore, $z \in b^k D_b^a \one + (b^{k+1}-b^k) D_b^a \Delta^s(\zero)$ if and only if
\begin{align*}
& \quad 0 \prec_{d+1} \frac{z_d}{b^{a_d}} - b^k \prec_d \frac{z_{d-1}}{b^{a_{d-1}}} - b^k \prec_{d-1} \ldots \prec_2 \frac{z_1}{b^{a_1}} - b^k \prec_1 b^{k+1}-b^k \\
\Longleftrightarrow & \quad b^k \prec_{d+1} \frac{z_d}{b^{a_d}} \prec_d \frac{z_{d-1}}{b^{a_{d-1}}} \prec_{d-1} \ldots \prec_2 \frac{z_1}{b^{a_1}} \prec_1 b^{k+1} \\
\Longleftrightarrow & \quad k \prec_{d+1} \log_b(z_d) - a_d \prec_d \ldots \prec_2 \log_b(z_1) - a_1 \prec_1 k+1,
\end{align*}
which holds if and only if $\phi(z) = (\log_b(z_1),\ldots,\log_b(z_d))^\intercal \in k \one + a + \Delta^s(\zero)$.
Here we also used that the logarithm $x \mapsto \log_b(x)$ is strictly increasing.
\end{proof}

\begin{example}
  The alcoved simplex $\conv \begin{pmatrix} 3 & 4 & 4 \\ 5 & 5 & 6 \end{pmatrix} = 2 \odot \conv \begin{pmatrix} 1 & 2 & 2 \\ 3 & 3 & 4 \end{pmatrix}$ maps to $7^2 \cdot \conv \begin{pmatrix} 7^1 & 7^2 & 7^2 \\ 7^3 & 7^3 & 7^4 \end{pmatrix} = 7^2 \cdot \left( \begin{pmatrix} 7^1 \\ 7^3 \end{pmatrix} + \conv \begin{pmatrix} 0 & 6\cdot 7 & 6 \cdot 7 \\ 0 & 0 & 6\cdot 7^3 \end{pmatrix} \right) = 7^2 \cdot \left( D_{7}^{(1,3)} \one + 6 \cdot D_{7}^{(1,3)} \cdot \conv \begin{pmatrix} 0 & 1 & 1 \\ 0 & 0 & 1 \end{pmatrix} \right)$ via $\exp_7$. 
\end{example}

Note that the proof of Lemma~\ref{lem:EhrhartSimplex} suggests that the tropical Ehrhart polynomial is close to a weighted version of the usual Ehrhart polynomial with weight function $z \mapsto b^z$.
Weighted Ehrhart polynomials have been studied, for instance, by Baldoni et al.~\cite{baldonietal2012computation} (they use polynomial weight functions but also discuss exponential weights).

\begin{example}
\label{ex:PickDelta}
By Pick's Theorem (cf.~\cite[Ch.~2.6]{beckrobins2007computing}) the Ehrhart function of a lattice polygon~$P \subseteq \R^2$ is given by
\[
 \shp{ kP \cap \Z^2 } = \vol(P) k^2 + \tfrac12 \shp{ \partial P \cap \Z^2 } k + 1. 
\]
Since $\vol(D_b^a \Delta(\zero)) = \tfrac12 b^{a_1+a_2}$ and $\shp{\partial D_b^a \Delta(\zero) \cap \Z^2} = b^{a_1} + b^{a_2} + b^{\min(a_1,a_2)}$, we use Lemma~\ref{lem:EhrhartSimplex} and we get the tropical Ehrhart polynomial of $\Delta(a)$, for each $a \in \Z_{\geq0}^2$:
\begin{align*}
\TL_{\Delta(a)}^b(k) &= \tfrac12 b^{a_1+a_2} (b^{k+1} - b^k)^2 + \tfrac12 (b^{a_1} + b^{a_2} + b^{\min(a_1,a_2)}) (b^{k+1} - b^k) + 1 \\
&= \tfrac12 (b-1)^2 b^{a_1+a_2} (b^k)^2 + \tfrac12 (b-1) (b^{a_1} + b^{a_2} + b^{\min(a_1,a_2)}) b^k + 1.
\end{align*}
\end{example}

The following is our desired precise version of Theorem~\ref{thm:tropEhrPoly}, building on the structure of the covector decomposition discussed in Section~\ref{sect:alcovedtriang}.
It expresses the tropical Ehrhart coefficients as signed and weighted sums of the classical Ehrhart coefficients of diagonally transformed alcoved simplices.

\begin{theorem}
\label{thm:tropEhrCoeffsRep}
The $i^{th}$ tropical Ehrhart coefficient of the tropical lattice polytope $P \subseteq \TT^d$ is given by
\begin{align}
c_i^b(P) &= \sum_{\substack{\Delta_\pi^s(a) \in \T_P\\ m=\dim(\Delta_\pi^s(a)) \geq i}} (-1)^{m-i} (b-1)^i c_i(D_b^a \widebar{\Delta^s_\pi(\zero)}),\label{eqn:cib-representation}
\end{align}
where $\overline{Q}$ denotes the closure of~$Q \subseteq \R^d$.
\end{theorem}

\begin{proof}
Every element of the alcoved triangulation $\T_P$ of~$P$, as discussed in Section~\ref{sect:alcovedtriang}, is of the form $\Delta_\pi^s(a)$, for some $s \in \{=,\leq,<\}^{d+1}$ and $a \in \Z_{\geq0}$.
Moreover we think of these alcoved simplices as being relatively open, that is, $s \in \{=,<\}^{d+1}$, since this yields a partition of~$P$ into these pieces.

Therefore, the tropical lattice point enumerator $\TL_P^b(k) = \shp{ (k \odot P) \cap \bLd }$ is the sum of the functions $\TL_{\Delta^s_\pi(a)}^b(k)$.
By Lemma~\ref{lem:EhrhartSimplex}, we have
\[
\TL_{\Delta^s_\pi(a)}^b(k) = \shp{ \left( k \odot \Delta^s_\pi(a) \right) \cap \bLd } = \shp{ \left( (b^{k+1}-b^k) D_b^a \Delta^s_\pi(\zero) \right) \cap \Z^d }.
\]
Now, $D_b^a \Delta^s_\pi(\zero)$ is a relatively open simplex all of whose vertices lie in $\Z^d$ and whose dimension is
\[
m = \#\{i : s_i = \text{`} < \text{'}\} - 1.
\]
Classical Ehrhart Theory on the standard lattice~$\Z^d$ (cf.~\cite[Ch.~3]{beckrobins2007computing}) implies that $\TL_{\Delta^s_\pi(a)}^b(k)$ agrees with a polynomial in $b^{k+1}-b^k$
of degree~$m$, whose coefficients depend on $\pi, a, s$, and $b$, but not on~$k$.
Thus, $\TL_{\Delta^s_\pi(a)}^b(k)$ agrees with a polynomial in $b^k$ for every $k \in \Z_{\geq0}$.
We conclude by observing that $\TL_P^b(k)$ as a sum of polynomials, is a polynomial in~$b^k$ as well.

In order to derive the stated formula for the $i^{th}$ tropical Ehrhart coefficient of~$P$, we write
\[
\shp{t D_b^a \widebar{\Delta_\pi^s(\zero)} \cap \Z^d} = \sum_{i=0}^m c_i(D_b^a \widebar{\Delta_\pi^s(\zero)}) t^i,
\]
and we observe that by Ehrhart reciprocity~\cite[Thm.~4.1]{beckrobins2007computing} we get
\begin{align*}
\shp{t D_b^a \Delta_\pi^s(\zero) \cap \Z^d} &= (-1)^m \shp{(-t) D_b^a \widebar{\Delta_\pi^s(\zero)} \cap \Z^d}\\
&= \sum_{i=0}^m (-1)^{m-i} c_i(D_b^a \widebar{\Delta_\pi^s(\zero)}) t^i.
\end{align*}
Substituting $t = (b-1)b^k$ and summing over all at least $i$-dimensional elements in $\T_P$ as described above finishes the proof.
\end{proof}

\subsection{First properties of tropical Ehrhart coefficients}

Here, we record two properties of tropical Ehrhart coefficients that go well in line with their classical counterparts.
We write $\P_{\TT,\L}^d$ for the family of tropical lattice polytopes in~$\TT^d$.

\goodbreak

\begin{proposition}
\label{prop:tropEhrCoeffs-props}
Let $P \in \P_{\TT,\L}^d$ and let $i \in \{0,1,\ldots,d\}$.
\begin{enumerate}[i)]
 \item \emph{(Homogeneity):} For every $\lambda \in \TT$, we have
 \[
 c_i^b(\lambda \odot P) = (b^{\lambda i}) \cdot c_i^b(P).
 \]

 \item \emph{(Valuation):} For every $b \in \N_{\geq2}$, the function $c_i^b(\cdot) : \P_{\TT,\L}^d \to \R$ is a \emph{valuation}, that is,
  \[c_i^b(P \cup Q) + c_i^b(P \cap Q) = c_i^b(P) + c_i^b(Q),\]
  for all $P,Q \in \P_{\TT,\L}^d$ such that $P \cup Q, P \cap Q \in \P_{\TT,\L}^d$.
\end{enumerate}
\end{proposition}

\begin{proof}
\romannumeral1): We use the representation of $c_i^b(P)$ provided in~\eqref{eqn:cib-representation}.
With the notation established around this identity, we observe that $\Delta_\pi^s(a) \subseteq P$ if and only if $\Delta_\pi^s(a+\lambda\one) \subseteq \lambda \odot P$.
Thus
\begin{align*}
c_i^b(\lambda \odot P) &= \sum_{\substack{\Delta_\pi^s(a+\lambda\one) \in \T_{\lambda\odot P}\\ m=\dim(\Delta_\pi^s(a+\lambda\one)) \geq i}} \!\!\!(-1)^{m-i}(b-1)^i c_i(D_b^{a+\lambda\one}\widebar{\Delta^s_\pi(\zero)})\\
&= \sum_{\substack{\Delta_\pi^s(a) \in \T_P\\ m=\dim(\Delta_\pi^s(a)) \geq i}} \!\!\!(-1)^{m-i}(b-1)^i c_i(b^\lambda \cdot D_b^a \widebar{\Delta^s_\pi(\zero)})\\
&= \sum_{\substack{\Delta_\pi^s(a) \in \T_P\\ m=\dim(\Delta_\pi^s(a)) \geq i}} \!\!\!(b^\lambda)^i (-1)^{m-i} (b-1)^i c_i(D_b^a \widebar{\Delta^s_\pi(\zero)}) = (b^{\lambda i}) \cdot c_i^b(P),
\end{align*}
because the classical $i^{th}$ Ehrhart coefficient $c_i(\cdot)$ is homogeneous of degree~$i$.

\romannumeral2): Clearly, the counting function $P \mapsto \shp{(k \odot P) \cap \Gamma_b^d}$ is a valuation, for every fixed $k \in \Z_{\geq0}$.
Therefore,
\[
\TL_{P \cup Q}^b(k) + \TL_{P \cap Q}^b(k) = \TL_{P}^b(k) + \TL_{Q}^b(k),
\]
and since every involved summand is a polynomial in $b^k$ of degree~$d$, the claim follows by comparing coefficients.
\end{proof}

\begin{example}
\label{ex:triangle+edge}
For $\ell \in \Z_{>0}$ and $k \in \Z_{\geq0}$, let $M = \left( \begin{array}{ccc} \ell-1 & \ell & k+\ell \\ 0 & 0 & k+1 \end{array} \right)$ and consider the tropical lattice polygon $P = \tconv(M)$ (see Figure~\ref{fig:differing+lower+vols}).
We aim to compute its first tropical Ehrhart coefficient $c_1^b(P)$.

Note, that~$P$ decomposes into the alcoved triangle $T=\Delta((\ell-1,0)^\intercal)$ and the segment $S=[(\ell,1)^\intercal,(k+\ell,k+1)^\intercal]$, which itself is decomposed into the alcoved segments $S_j = (\ell+j-1,j)^\intercal + [\zero,\one]$, for $1 \leq j \leq k$.
Hence, by the valuation property in Proposition~\ref{prop:tropEhrCoeffs-props}, and the fact that the occurring intersections are zero-dimensional, we get by Lemma~\ref{lem:EhrhartSimplex} and Example~\ref{ex:PickDelta}
\begin{align*}
c_1^b(P) &= c_1^b(T) + c_1^b(S_1) + \ldots + c_1^b(S_k)\\
&= \frac12(b-1)(b^{\ell-1}+2) + (b-1)(b+\ldots+b^k).
\end{align*}
\end{example}

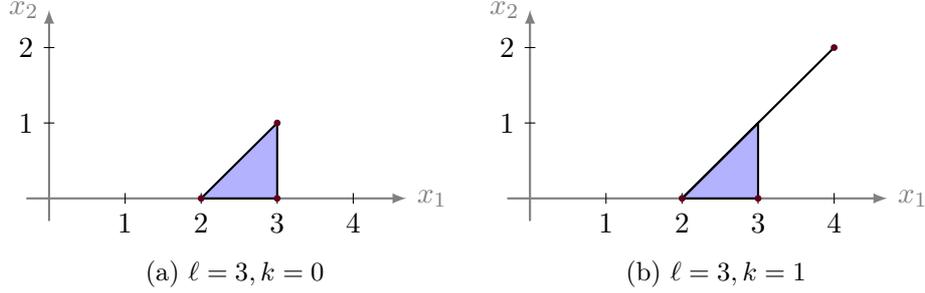
\begin{figure}[ht]
  \begin{subfigure}[h]{0.49\textwidth}
  \begin{tikzpicture}

  \Koordinatenkreuz{-0.3}{4.7}{-0.3}{2.5}{$x_1$}{$x_2$}

  \foreach \x in {1,...,4}
    \draw[shift={(\x,0)}] (0pt,2pt) -- (0pt,-2pt) node[below] {$\x$};

    \foreach \y in {1,2}
    \draw[shift={(0,\y)}] (2pt,0pt) -- (-2pt,0pt) node[left] {$\y$};

  
  \coordinate (a) at (2,0);
  \coordinate (b) at (3,0);
  \coordinate (c) at (3,1);


    \draw[black,thick, fill=blue!30] (2,0) -- (3,0) -- (3,1) -- cycle;

  \node[Endpoint] at (a){};
  \node[Endpoint] at (b){};
  \node[Endpoint] at (c){};

  \end{tikzpicture}
  \caption{$\ell = 3, k = 0$}
  \end{subfigure}
  \begin{subfigure}[h]{0.49\textwidth}
  \begin{tikzpicture}

  \Koordinatenkreuz{-0.3}{4.7}{-0.3}{2.5}{$x_1$}{$x_2$}

  \foreach \x in {1,...,4}
    \draw[shift={(\x,0)}] (0pt,2pt) -- (0pt,-2pt) node[below] {$\x$};

    \foreach \y in {1,2}
    \draw[shift={(0,\y)}] (2pt,0pt) -- (-2pt,0pt) node[left] {$\y$};

  
  \coordinate (a) at (2,0);
  \coordinate (b) at (3,0);
  \coordinate (c) at (4,2);

  \coordinate (help) at (3,1);

    \draw[black,thick, fill=blue!30] (a) -- (b) -- (help) -- cycle -- (c);

  \node[Endpoint] at (a){};
  \node[Endpoint] at (b){};
  \node[Endpoint] at (c){};

  \end{tikzpicture}
  \caption{$\ell = 3, k = 1$}
  \end{subfigure}
  \caption{The tropical lattice polygons in Example~\ref{ex:triangle+edge}.}
  \label{fig:differing+lower+vols}
\end{figure}

\section{Tropical Volume from Tropical Lattice Points}
\label{sec:trop+vol+lattice+points}

\subsection{A novel concept of tropical volume}

Based on the polynomiality of the counting function $k \mapsto \TL_P^b(k)$ established in Section~\ref{sec:tropEhrpolys}, we can now easily build up an analogy to the classical setting and define a novel volume concept for tropical (lattice) polytopes:
If $P \subseteq \R^d$ is a classical lattice polytope, that is, with respect to $(+,\cdot)$, and with Ehrhart polynomial $\shp{ k P \cap \Z^d } = \sum_{i=0}^d c_i(P) k^i$, then by properties of the Lebesgue-measure one obtains
\[
\vol(P) = \lim_{k \to \infty} \frac{\shp{ k P \cap \Z^d }}{k^d} = c_d(P).
\]
Note that $k^d = \shp{ k \cdot [0,1)^d \cap \Z^d }$, that is, $k^d$ is the number of lattice points in the~$k^{th}$ dilate of the standard fundamental cell of $\Z^d$.
The tropicalization of this statement is given by
\[
\shp{ \left(k \odot [-\infty,0)^d\right) \cap \bLd } = \shp{[-\infty,k)^d \cap \bLd } = (b^k)^d.
\]
Thus writing $\TL_P^b(k) = \sum_{i=0}^d c_i^b(P) (b^k)^i$ for a \emph{tropical} lattice polytope~$P \subseteq \TT^d$, we obtain
\[
c_d^b(P) = \lim_{k \to \infty} \frac{\shp{(k \odot P) \cap \bLd }}{(b^k)^d}.
\]
This observation motivates the following definition.

\begin{definition}[Tropical $b$-volume]
Let $P \subseteq \TT^d$ be a tropical lattice polytope and let $b \in \N_{\geq 2}$.
Define the \emph{tropical $b$-volume} $\tlvol^b(P)$ of $P$ as the leading coefficient $c_d^b(P)$ of its tropical Ehrhart polynomial $\TL_P^b(k)$.\end{definition}

In view of Theorem~\ref{thm:Ehrhart-max-times}, we get that $\tlvol^b(P) = \vol(\exp_b(P))$.

\begin{example}
\label{ex:trop-b-volume}
Consider the translated standard alcoved simplex $\Delta(a) \subseteq \TT^d$, where $a \in \Z_{\geq0}^d$.
In view of~\eqref{eqn:cib-representation}, its tropical $b$-volume equals
\[
\tlvol^b(\Delta(a)) = c_d^b(\Delta(a)) = (b-1)^d c_d(D_b^a \Delta(\zero)) = \tfrac1{d!} (b-1)^d b^{a_1+\ldots+a_d}.
\]
\end{example}

\noindent In view of this example, we see that the tropical $b$-volume of a tropical lattice polytope~$P$ equals the sum of $\tlvol^b(\Delta_\pi(a)) = \tfrac1{d!} (b-1)^d b^{a_1+\ldots+a_d}$, where $\Delta_\pi(a) \in \T_P$, for some $a \in \Z^d$ and some permutation $\pi \in S_d$.

As a consequence $\tlvol^b(P)$ is a polynomial, seen as a function of $b \in \N_{\geq2}$. Hence, we can apply the \emph{logarithm-map}
\[
\Log{f} := \lim_{b \to \infty} \log_b \,\card{f(b)}
\]
to $\tlvol^b(P)$ and arrive at a tropical volume concept for~$P$ which is independent of any additional parameter.
Note that $\Log{f}$ does not exist for all functions $f:\R \to \R$, however, we only apply it to the rational functions~$c_d^b(P)$ and $c_{d-1}^b(P)$ (cf.~Lemma~\ref{lem:second+highest+trop+coeff}).

\begin{definition}[Tropical barycentric volume] \label{def:tropical+volume+from+lattice}
Let $P \subseteq \TT^d$ be a tropical lattice polytope.
We define its \emph{tropical barycentric volume} by
\[
\tlvol(P) := \max\{ a_1+\ldots+a_d + d : a \in \Z^d\text{ such that } \Delta_\pi(a) \in \T_P \}.
\]
\end{definition}

\noindent Note that the alcoved simplices $\Delta_\pi(a)$ appearing in this definition are full-dimensional.
The tropical barycentric volume can be interpreted in terms of the $d$-trunk from Definition~\ref{def:i+trunk}.

\begin{proposition} \label{prop:tropical+volume+maximal+sum}
  The tropical barycentric volume of a tropical lattice polytope~$P$ equals
\[
\tlvol(P) = \max_{x \in \Tr_d(P)} \sum_{i=1}^d x_i . 
\]
\end{proposition}

\begin{proof}
Let $P \subseteq \TT^d$ be a tropical lattice polytope.
The $d$-trunk of $P$ is the union of all full-dimensional alcoved simplices $\Delta_\pi(a) \in \T_P$.
The maximal point with respect to the linear functional $\one$ of such a simplex $\Delta_\pi(a)$ is given by $a + \one$, so its coordinate sum equals $a_1+\ldots+a_d+d$.
The claim follows by observing that the maximal point of~$\Tr_d(P)$ is the maximal point of a suitable simplex~$\Delta_\pi(a)$.
\end{proof}

This interpretation allows us to extend the definition to tropical polytopes whose vertices are not necessarily lattice points. 

\begin{definition} \label{def:tropical+volume}
  The \emph{tropical barycentric volume} $\tlvol(P)$ of a tropical polytope $P \subseteq \TT^d$ is defined as $\max_{x \in \Tr_d(P)} \sum_{i=1}^d x_i$.
\end{definition}

In the following we use the compact and more convenient notation $\one^\intercal x = \sum_{i=1}^d x_i$.
The tropical barycentric volume has a particularly nice form, if the tropical polytope is pure.
For this, we need the notion of the \emph{tropical barycenter}, which is the componentwise maximal point of a tropical polytope.
This point exists and it is moreover unique due to the definition of tropical convex combinations.
Proposition~\ref{prop:dtrunk+tropical+convex} implies the following.

\begin{corollary}
  The tropical barycentric volume is the sum of the coordinates of the barycenter of its $d$-trunk.
  In particular, the tropical barycentric volume of a pure tropical polytope is the sum of the coordinates of its barycenter.
\end{corollary}

This observation also explains our choice to call $\tlvol(\cdot)$ the tropical \emph{barycentric} volume.

\begin{example}
  The tropical unit cube in $\TT^d$ is given as the Cartesian product $[-\infty,0]^d$.
  It can be written as the tropical convex hull of the points $(-\infty,\ldots,-\infty)^\intercal$, $(-\infty,0,\ldots,0)^\intercal, \ldots, (0,\ldots,0,-\infty)^\intercal$.
  Its tropical barycentric volume is $0$, the tropical one.   
\end{example}

\subsection{Properties of the tropical barycentric volume}
\label{sec:properties+tropical+volume}

We now collect basic properties of the tropical barycentric volume, exhibiting the close analogy to the Euclidean volume.
To this end, we need to introduce some notation.

We write $r^{\odot k} := \underbrace{r \odot \ldots \odot r}_{k\text{ times}}$ for tropical exponentiation.
Furthermore, let~$\P_\TT^d$ be the family of tropical polytopes in~$\TT^d$.
For $z \in \TT^d$, we consider the diagonal matrix $\diag(z_1, \ldots, z_d) \in \TT^{d \times d}$, and for an arbitrary permutation in the symmetric group on $d$ elements, let $\Sigma$ be the corresponding tropical permutation matrix.
The matrices of the form $\diag(z_1,\ldots,z_d) \odot \Sigma$ form the group $\Pi_d$ of \emph{scaled permutation matrices}.
We denote the subgroup of the matrices with $\bigodot_{i \in [d]} z_i = 0$, that is, those with tropical determinant equal to~$0$, by $\cR_d$. 

\goodbreak

\begin{proposition} \quad
\label{prop:tlvol-props}
\begin{enumerate}[i)]
 \item \emph{(Monotonicity):} For every $P,Q \in \P_\TT^d$ with $P \subseteq Q$, we have
 \[\tlvol(P) \leq \tlvol(Q).\]

 \item \emph{(Valuation):} $\tlvol : \P_\TT^d \to \TT$ is a \emph{valuation} in the sense that
 \[\tlvol(P) \oplus \tlvol(Q) = \tlvol(P \cup Q) \oplus \tlvol(P \cap Q),\]
 for every $P,Q \in \P_\TT^d$ such that $P \cup Q, P \cap Q \in \P_\TT^d$.

\item \emph{(Rotation invariance):}
  For $M \in \cR_d$ and $P \in \P_\TT^d$, we have \[\tlvol(M \odot P) = \tlvol(P).\]

 \item \emph{(Homogeneity):} For every $\lambda \in \TT$ we have \[\tlvol(\lambda \odot P) = \lambda^{\odot d} \odot \tlvol(P).\]

 \item \emph{(Non-singularity):} $\tlvol(P) = -\infty$ if and only if $\Tr_d(P) = \emptyset$.
\end{enumerate}
\end{proposition}

\noindent We will prove a more general statement in Proposition~\ref{prop:tlvoli-props}.

\begin{remark}
  Property~\romannumeral2) in Proposition~\ref{prop:tlvol-props} actually holds in a stronger form.
  Indeed, $\tlvol : \P_\TT^d \to \TT$ is an \emph{idempotent measure}, which means that $\max\left\{\tlvol(P),\tlvol(Q)\right\} = \tlvol(P \cup Q)$.
For a thorough investigation of idempotent measures we refer the reader to Akian~\cite{Akian:1999}.
\end{remark}

The Euclidean volume $\vol(\cdot)$ is multiplicative with respect to taking Cartesian products, that is, for any ordinary polytopes $P \subseteq \R^d$ and $Q \subseteq \R^e$ we have $\vol(P \times Q) = \vol(P) \cdot \vol(Q)$.
Again, the tropical barycentric volume $\tlvol(\cdot)$ exhibits an analogous behavior.

\begin{proposition}
\label{prop:multiplicativity-tlvol}
Let $P \in \P_\TT^d$ and $Q \in \P_\TT^e$.
Then, $P \times Q \in \P_\TT^{d+e}$ and
\[
\tlvol(P \times Q) = \tlvol(P) \odot \tlvol(Q).
\]
\end{proposition}

\begin{proof}
The fact that $P \times Q$ is a tropical polytope when $P$ and $Q$ are, was proven in~\cite[Thm.~2]{DevelinSturmfels:2004}.
The claimed identity is based on the observation that taking the trunk commutes with taking Cartesian products, more precisely
\begin{align}
\Tr_{d+e} (P \times Q) = \Tr_d (P) \times \Tr_e (Q).\label{eqn:trunk-product}
\end{align}
Indeed, for any face $F \in \F_{P \times Q}$ that is contained in the $(d+e)$-trunk, there is a face $G \in \F_{P \times Q}$ with $F \subseteq G$ and $\dim(G) = d+e$.
Since every face of a product of polytopal complexes is a product of faces of the factors, we find $G_P \in \F_P$ and $G_Q \in \F_Q$ such that $G = G_P \times G_Q$, and since $\dim(G_P) + \dim(G_Q) = d+e$, we have $\dim(G_P)=d$ and $\dim(G_Q)=e$.
Therefore, writing $F = F_P \times F_Q$ for some $F_P \in \F_P$ and $F_Q \in \F_Q$, we obtain $F_P \subseteq G_P$ and $F_Q \subseteq G_Q$ and thus $F \in \Tr_d(P) \times \Tr_d(Q)$.
As all these arguments can be reversed, the relation~\eqref{eqn:trunk-product} follows.

With this information, we now have
\begin{align*}
\tlvol(P \times Q) &= \max_{x \in \Tr_{d+e}(P \times Q)} \one^\intercal x = \max_{(y,z) \in \Tr_d(P) \times \Tr_e(Q)} \one^\intercal(y,z)\\
&=\max_{y \in \Tr_d(P)} \one^\intercal y  + \max_{z \in \Tr_e(Q)} \one^\intercal z = \tlvol(P) \odot \tlvol(Q).\qedhere
\end{align*}
\end{proof}

\begin{example}
  A \emph{tropical prism} is the Cartesian product of a tropical polytope and a $1$-dimensional tropical polytope in $\TT$.
  As each $1$-dimensional tropical polytope is pure, its tropical barycentric volume is just the maximal point.
\end{example}

\subsection{Tropical volume revisited}
\label{sec:trop+vol+revisited}

We compare our volume notion with the two volume concepts introduced in~\cite{DepersinGaubertJoswig:2017}. 

\subsubsection{Second highest determinant}
\label{sssect:tvol}

For a matrix $A \in \TT^{d \times d}$, and a permutation $\sigma \in S_d$ we write $\tdet^\sigma(A) = \bigoplus_{\pi \in S_d \setminus \{\sigma\}} \bigodot_{i=1}^d a_{i,\pi(i)}$.
The tropical volume concept introduced in~\cite{DepersinGaubertJoswig:2017} can then be defined by
\[
\tvol(A) = \card{\tdet(A) - \tdet^\sigma(A)} \enspace ,
\]
where $\sigma \in S_d$ is a permutation at which $\tdet(A)$ is attained.
For the sake of distinction, we call this the \emph{tropical permutation volume} of~$A$.
This notion is motivated from an `energy gap' in statistical physics used in~\cite{KosowskyYuille:1994}. 
A property that distinguishes $\tvol(\cdot)$ from $\tlvol(\cdot)$ is that it is translation invariant in the classical sense, that is, if we write $v+A$ for the matrix that arises from $A$ after adding the vector $v \in \R^d$ to each column of~$A$, then $\tvol(v+A) = \tvol(A)$.
However, the homogeneity of $\tlvol(\cdot)$ described in Proposition~\ref{prop:tlvol-props}~\romannumeral3) shows that the two volume concepts are incomparable.
On the other hand, there always exists a scalar $\lambda \in \R$ such that $\tvol(A) = \tlvol(\lambda\one + P)$, where $P=\tconv(A)$.

\smallskip

As $\tvol(\cdot)$ is only defined for square matrices, we discuss potential extensions to rectangular matrices. 
The metric quantities in Definition~\ref{def:tropical+volume+from+lattice} and Definition~\ref{def:dequantized+volume} below are extended from a local measure to a global measure by taking a maximum, over points or submatrices. 
  Applying this idea to $\tvol(\cdot)$ suggests to extend the tropical permutation volume to rectangular matrices $A \in \TT^{d \times m}$ with $d \leq m$ by setting
  \begin{equation} \label{eq:maximal+tvol}
    \tvol_{\text{max-sub}}(A) = \max_{J \in \binom{[m]}{d}} \tvol(A_J) \enspace .
  \end{equation}
  This definition keeps the desirable property that the tropical permutation volume is zero if the tropical convex hull is lower-dimensional.

  In the study of tropical principal component analysis, the notion $\tvol(\cdot)$ is also discussed in~\cite[\S\,3.1]{YoshidaZhangZhang:2019}.
  They propose an extension to rectangular matrices in terms of a sum of tropical distances
  \begin{equation} \label{eq:best+fit+hyperplane}
    \tvol_{\text{best-fit}}(A) = \min_{z \in \TT^d} \sum_{j \in [m]} \operatorname{d_{Tr}}(\mathcal{H}(z),A^{(j)}) \enspace ,
  \end{equation}
  where $A^{(j)}$ is the $j$-th column of $A$, 
  \[
  \mathcal{H}(z) = \left\{x \in \TT^d : \exists\,i \neq j\text{ such that } x_i + z_i = x_j + z_j \geq x_{\ell} + z_{\ell}, \forall \ell \in [d] \right\}
  \]
  is the tropical hyperplane defined by $z$, and
  \[
  \operatorname{d_{Tr}}(v,w) = \max\left\{|v_i - w_i - v_j + w_j| \colon 1 \leq i,j \leq d\right\} 
  \]
  is the generalized Hilbert projective metric (cf.~\cite{CohenGaubertQuadrat:2004}).

  We propose a further investigation of the quantities defined in~\eqref{eq:maximal+tvol} and~\eqref{eq:best+fit+hyperplane}.

\subsubsection{Tropical dequantized volume}

The next concept defines a volume of a tropical polytope in terms of volumes of associated polytopes over Puiseux series as discussed in Section~\ref{sec:versions+convexity}.
Since $\pseries{\R}{t}$ is a real closed field, concepts like convex hull, polytope, and Euclidean volume can be defined analogously to the setting over~$\R$.
We use the shorthand notation $\vol \mathbf{M} := \vol(\conv(\mathbf{M}))$ for a matrix $\mathbf{M} \in \pseries{\R}{t}^{d \times m}$.

\begin{definition}[{\cite{DepersinGaubertJoswig:2017}}] \label{def:dequantized+volume}
  For a matrix $M \in \TT^{d \times m}$, the \emph{tropical (upper) dequantized volume} of $M$ is defined by
\[
\qtvol^+(M) := \sup\left\{\val \vol \mathbf{M} : \mathbf{M} \in \pseries{\R}{t}^{d \times m}, \val \mathbf{M} = M \right\}.
\]
\end{definition}

Depersin, Gaubert \& Joswig~\cite[Thm.~13]{DepersinGaubertJoswig:2017} gave an interpretation of the tropical dequantized volume in terms of the maximal tropical minor.
More precisely, they prove that for every $M \in \TT^{d \times m}$ we have
\begin{align}
\qtvol^+(M) = \max_{J \in \binom{[m]}{d}} \tdet(M_J).\label{eqn:DGJ}
\end{align}
The idea behind this formula is that the volume is essentially dominated by the maximal determinant of a simplex contained in a polytope. 
A close inspection of their proof shows that one can restrict to monomial lifts with integral coefficients, that is,
\[
\qtvol^+(M) = \sup\left\{\val \vol \mathbf{M} : \val \mathbf{M} = M\text{ with } \mathbf{M} \in \Z\{\!\!\{t\}\!\!\}^{d \times m} \text{ monomial}\right\},
\]
where $\mathbf{M}$ is called \emph{monomial} if $\mathbf{M}_{ij} = \alpha_{ij} t^{m_{ij}}$, for every $i,j$.

Based on~\eqref{eqn:DGJ}, we obtain that the tropical dequantized volume is an upper bound on the tropical barycentric volume.

\begin{theorem}
\label{thm:tvol_qtvol}
Let $M \in \TT\N^{d \times m}$ and let $P = \tconv(M)$ be the corresponding tropical lattice polytope.
Then,
\[
\tlvol(P) \leq \qtvol^+(M).
\]
\end{theorem}

\noindent This inequality is a special case of Theorem~\ref{thm:tlvoli_tmi} that we prove later.
Note that for every $\ell \in \N_{\geq2}$, there exists a matrix $M$ such that $\tlvol(P) = 2$ and $\qtvol^+(M) = \ell$.
An example is given by $M = \left( \begin{array}{ccc} 0 & 0 & \ell-1 \\ 0 & 1 & \ell-1 \end{array} \right)$.
The case $\ell=4$ is depicted in Figure~\ref{fig:latpoly+3}.

A characterization of the equality case follows right from Definition~\ref{def:tropical+volume}.

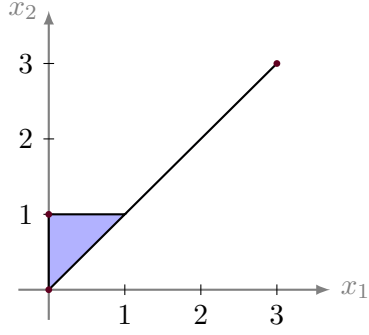
\begin{figure}
  \begin{tikzpicture}

  \Koordinatenkreuz{-0.4}{3.7}{-0.4}{3.7}{$x_1$}{$x_2$}

    \foreach \x in {1,2,3}
    \draw[shift={(\x,0)}] (0pt,2pt) -- (0pt,-2pt) node[below] {$\x$};

    \foreach \y in {1,2,3}
    \draw[shift={(0,\y)}] (2pt,0pt) -- (-2pt,0pt) node[left] {$\y$};

  \coordinate (a) at (0,0);
  \coordinate (b) at (0,1);
  \coordinate (c) at (3,3);

  \coordinate (origin) at (0,0);


  \draw[black,thick, fill=blue!30] (1,1) -- (0,0) -- (0,1) -- cycle -- (3,3);

  \node[Endpoint] at (a){};
  \node[Endpoint] at (b){};
  \node[Endpoint] at (c){};
  
\end{tikzpicture}
  \caption{A canonical tropical lattice polygon $P$ with $\tlvol(P) = 2$ and $\qtvol^+(P) = 4$.}
  \label{fig:latpoly+3}
\end{figure}

\begin{proposition}
Let $M \in \TT\N^{d \times m}$ and let $P = \tconv(M)$.
Then
\[
\tlvol(P) = \qtvol^+(M)
\]
if and only if the tropical barycenter of~$P$ is contained in~$\Tr_d(P)$.
\end{proposition}

\begin{corollary}
\label{cor:pure}
For pure tropical lattices polytopes $P = \tconv(M)$, the magnitudes $\qtvol^+(M)$ and $\tlvol(P)$ agree.
\end{corollary}

Although this shows that the two tropical volume concepts $\tlvol(\cdot)$ and $\qtvol^+(\cdot)$ are closely related, they are inherently different.
For example, the multiplicativity of $\tlvol(\cdot)$ proved in Proposition~\ref{prop:multiplicativity-tlvol} is not shared by $\qtvol^+(\cdot)$ in general.
For instance, let $M=\left(\begin{array}{ccc}0 & 1 & \ell\\ 0 & 0 & \ell\end{array}\right)$ and $N=(0 \ 1)$.
Then, $\qtvol^+(M) = \ell + 1$, $\qtvol^+(N) = 1$, and $\qtvol^+(M \times N) = 2 \ell + 1$, where $M \times N=\left(\begin{array}{cccccc}0 & 1 & \ell & 0 & 1 & \ell\\ 0 & 0 & \ell & 0 & 0 & \ell\\ 0 & 0 & 0 & 1 & 1 & 1\end{array}\right)$ represents the defining matrix of the product of the tropical polytopes $\tconv(M) \subseteq \TT^2$ and $\tconv(N) \subseteq \TT$.

\section{Metric estimates for tropical polytopes}
\label{sec:metric+estimates+lower+volumes}

We here study generalizations of $\tlvol(\cdot)$ and $\qtvol^+(\cdot)$ to lower-dimen\-sional parts and see how they relate to each other and to the tropical Ehrhart coefficients.

\subsection{Lower-dimensional tropical volumes and their properties}

Using the intuition gained from Proposition~\ref{prop:tropical+volume+maximal+sum}, we define finer measures to study volumetric properties of tropical polytopes. 
We will see that the following is both a useful and reasonable concept of lower-dimensional tropical volume measures.

\begin{definition}[Tropical barycentric $i$-volumes]
Let $P \subseteq \TT^d$ be a tropical polytope and let $i \in [d]$.
We define the \emph{tropical upper barycentric $i$-volume} and the \emph{tropical lower barycentric $i$-volume} of~$P$ by
\[
\tlvol_i^+(P) := \max_{x \in \Tr_i(P)} \,\max\left\{ v^\intercal x : v \in \{0,1\}^d, \one^\intercal v = i \right\}
\]
and
\[
\tlvol_i^-(P) := \max_{x \in \Tr_i(P)} \,\min\left\{ v^\intercal x : v \in \{0,1\}^d, \one^\intercal v = i \right\},
\]
respectively.
\end{definition}

\begin{example}
  We consider again the tropical polytope $P$ from Example~\ref{ex:non+connected+trunk}.
  As $P$ is $2$-dimensional, we get
  \[
  \tlvol_4^+(P) = \tlvol_4^-(P) = \tlvol_3^+(P) = \tlvol_3^-(P) = -\infty \enspace .
  \]
  Using the explicitly given pseudovertices, we obtain $\tlvol_2^+(P) = 18$ (attained at each point of the $2$-trunk), $\tlvol_2^-(P) = 2$ (attained at $p$ and $q$), $\tlvol_1^+(P)  = 9$ (attained at each point of the $1$-trunk), $\tlvol_1^-(P) = 9$ (attained at $r$).
\end{example}

When we write $\tlvol_i^\pm(\cdot)$, we refer to both the upper and the lower tropical barycentric $i$-volume simultaneously.
Each tropical barycentric $i$-volume comes with its own natural properties analogous to those of $\tlvol(\cdot)$ stated in Proposition~\ref{prop:tlvol-props}.
For the rotation invariance we need the following refined subsets of scaled permutation matrices (see Section~\ref{sec:properties+tropical+volume}):
\[
\cR_{d,i}^+ := \left\{\diag(z_1,\ldots,z_d) \odot \Sigma : \max\big\{ v^\intercal z : v \in \{0,1\}^d, \one^\intercal v = i \big\} = 0 \right\} \subseteq \Pi_d,
\]
and
\[
\cR_{d,i}^- := \left\{\diag(z_1,\ldots,z_d) \odot \Sigma : \min\big\{ v^\intercal z : v \in \{0,1\}^d, \one^\intercal v = i \big\} = 0 \right\} \subseteq \Pi_d.
\]
We retrieve $\cR_d = \cR_{d,d}^\pm$ as a special case.

\begin{example}
  Note that for $i < d$, these subsets do not necessarily form a group as the product
  \[
  \begin{pmatrix} -\infty & -\infty & 2 \\ -3 & -\infty & -\infty \\ -\infty & -2 & -\infty \end{pmatrix} \odot \begin{pmatrix} -1 & -\infty & -\infty \\ -\infty & -2 & -\infty \\ -\infty & -\infty & 1 \end{pmatrix} = \begin{pmatrix} -\infty & -\infty & 3 \\ -4 & -\infty & -\infty \\ -\infty & -4 & -\infty \end{pmatrix} \not\in \cR_{3,2}^+ 
  \]
  shows.
  An arbitrary matrix in $\cR_{3,1}^-$ is
  \[
  \begin{pmatrix} 0 & -\infty & -\infty \\ -\infty & -\infty & 4 \\ -\infty & 5 & -\infty \end{pmatrix} \enspace .
  \]
\end{example}

Since $\tlvol_d^\pm(P) = \tlvol(P)$, the proof of the following properties also proves Proposition~\ref{prop:tlvol-props}.

\begin{proposition}
\label{prop:tlvoli-props}\ 
\begin{enumerate}[i)]
 \item \emph{(Monotonicity):} For every $P, Q \in \P_\TT^d$ with $P \subseteq Q$, we have 
 \[\tlvol_i^\pm(P) \leq \tlvol_i^\pm(Q).\]

 \item \emph{(Idempotency):} For every $P,Q \in \P_\TT^d$ such that $P \cup Q \in \P_\TT^d$, we have
 \[\tlvol_i^\pm(P) \oplus \tlvol_i^\pm(Q) = \tlvol_i^\pm(P \cup Q).\]

 \item \emph{(Rotation invariance):} For every $P \in \P_\TT^d$ and every $M \in \cR_{d,i}^\pm$, we have \[\tlvol_i^\pm(M \odot P) = \tlvol_i^\pm(P).\]

 \item \emph{(Homogeneity):} For every $\lambda \in \TT$ we have \[\tlvol_i^\pm(\lambda \odot P) = \lambda^{\odot i} \odot \tlvol_i^\pm(P).\]

 \item \emph{(Non-singularity):} $\tlvol_i^\pm(P) = -\infty$ if and only if $\Tr_i(P) = \emptyset$.

\end{enumerate}
\end{proposition}

\begin{proof}
\romannumeral1): If $P \subseteq Q$, then $\Tr_i(P) \subseteq \Tr_i(Q)$.
 Thus $\tlvol_i^\pm(P) \leq \tlvol_i^\pm(Q)$.

\romannumeral2): If $P \cup Q \in \P_\TT^d$, then $\Tr_i(P \cup Q) = \Tr_i(P) \cup \Tr_i(Q)$ from which the claimed identity follows.

\romannumeral3): Let $M = \diag(z_1,\ldots,z_d) \odot \Sigma \in \cR_{d,i}^-$.
By definition
\begin{align*}
 \tlvol_i^-(M \odot P) &= \max_{x \in \Tr_i(\Sigma\odot P + z)} \,\min\left\{ v^\intercal x : v \in \{0,1\}^d, \one^\intercal v = i \right\}\\
 &= \tlvol_i^-(\Sigma\odot P) + \min\left\{ v^\intercal z : v \in \{0,1\}^d, \one^\intercal v = i \right\}\\
 &= \tlvol_i^-(P).
\end{align*}
 The proof for $\tlvol_i^+$ and matrices $M \in \cR_{d,i}^+$ is analogous.

\romannumeral4): By definition
 \begin{align*}
 \tlvol_i^-(\lambda \odot P) &= \max_{x \in \Tr_i(P + \lambda\one)} \,\min\left\{ v^\intercal x : v \in \{0,1\}^d, \one^\intercal v = i \right\}\\
 &= \tlvol_i^-(P) + \lambda i.
 \end{align*}
 Again, the proof for $\tlvol_i^+$ is analogous.

\romannumeral5): Immediate from the definition.
\end{proof}

\noindent It is easy to check that, since $\Tr_1(P) = P$, we have
\[
\tlvol_1^+(P) = \max_{1 \leq i \leq d} \tlvol(\pi_i(P)),
\]
where $\pi_i:\R^d \to \R$ is the projection onto the $i^{th}$ coordinate.
This raises the question whether the tropical upper barycentric $i$-volumes admit a tropical analog of the integral representation formula for the intrinsic volumes (or quermassintegrals) of an ordinary polytope (see~\cite{schneider1993convex} for definition and properties).
Roughly speaking these formulae show that the $i^{th}$ intrinsic volume is the average of the volumes of the $i$-dimensional projections of the given polytope (cf.~\cite[Thm.~19.3.2]{buragozalgaller1988geometric} for details).

\begin{question}
Is it true that for every tropical polytope $P \subseteq \TT^d$ and every index $i \in [d]$ we have
\[
\tlvol_i^+(P) = \max_{J \in \binom{[d]}{i}} \tlvol(\pi_J(P)),
\]
where $\pi_J:\R^d \to \R^{\card{J}}$ is the projection onto the coordinates indexed by~$J$ ?
\end{question}

Note that an analogous result cannot hold for the tropical lower barycentric $i$-volumes.
Even for $i=1$, the valid inequality
\[
\tlvol_1^-(P) \leq \min_{1 \leq i \leq d} \tlvol(\pi_i(P))
\]
can be strict.

\subsection{Estimates on tropical Ehrhart coefficients}

In this part, we argue how the tropical barycentric $(d-1)$-volumes can be used to estimate the second highest tropical Ehrhart coefficient.
To this end, let $Q \subseteq \R^d$ be an $m$-dimensional classical lattice polytope, with $m \leq d$.
The \emph{relative volume} of $Q$ is defined as
\[
\rvol(Q) := \frac{\vol_m(Q)}{\det(\Z^d \cap \aff(Q))} = \lim_{t \to \infty} \frac{1}{t^m} \cdot \shp{ tQ \cap \Z^d },
\]
where $\vol_m(Q)$ denotes the Lebesgue measure in $\aff(Q)$.
Of course, if $m=d$, then $\rvol(Q) = \vol(Q)$.
Let us record a well-known result from Ehrhart theory (cf.~\cite[Sect.~5.4]{beckrobins2007computing}): The highest coefficient of the Ehrhart polynomial $\shp{kQ \cap \Z^d} = \sum_{i=0}^m c_i(Q) k^i$ of $Q$, equals its relative volume, that is, $c_m(Q) = \rvol(Q)$.

The second highest tropical Ehrhart coefficient $c_{d-1}^b(P)$ of a tropical lattice polytope $P \subseteq \TT^d$ admits a more convenient representation than the signed sum in~\eqref{eqn:cib-representation}.
To this end, recall that an element in $\T_P$ is called \emph{maximal} if it is not properly contained in another element of~$\T_P$.

\begin{lemma}
\label{lem:second+highest+trop+coeff}
Let $P \subseteq \TT^d$ be a tropical lattice polytope.
Then,
\[
c_{d-1}^b(P) = \sum_{\substack{\Delta_\pi^s(a) \in \T_P\\ \dim(\Delta_\pi^s(a)) = d-1}} \delta(\Delta_\pi^s(a)) \cdot (b-1)^{d-1} \rvol(D_b^a\Delta^s_\pi(\zero)),
\]
where $\delta(\Delta_\pi^s(a)) = 1$, if $\Delta_\pi^s(a)$ is maximal, $\delta(\Delta_\pi^s(a)) = 0$, if $\Delta_\pi^s(a) \subseteq \acirc{P}$, and $\delta(\Delta_\pi^s(a)) = \frac12$, otherwise.
\end{lemma}

\begin{proof}
Specializing Theorem~\ref{thm:tropEhrCoeffsRep} to $i=d-1$ and in view of the remarks above, we have
\begin{align*}
c_{d-1}^b(P) &= \sum_{\substack{\Delta_\pi^s(a) \in \T_P\\ \dim(\Delta_\pi^s(a)) = d-1}} (b-1)^{d-1} \rvol(D_b^a\Delta^s_\pi(\zero))\\
&\phantom{=} - \sum_{\substack{\Delta_\pi^s(a) \in \T_P\\ \dim(\Delta_\pi^s(a)) = d}} (b-1)^{d-1} c_{d-1}(D_b^a \widebar{\Delta^s_\pi(\zero)}).
\end{align*}
The classical description of the second highest Ehrhart coefficient of a lattice polytope (cf.~\eqref{eqn:classical-c-d-1}) implies that for $d$-dimensional alcoved simplices $\Delta_\pi^s(a) \in \T_P$, we have $c_{d-1}(D_b^a \widebar{\Delta^s_\pi(\zero)}) = \frac12 \sum_{F} \rvol(F)$, where the sum runs over the facets $F$ of $D_b^a \widebar{\Delta^s_\pi(\zero)}$.
Observe that each of these facets corresponds to a $(d-1)$-dimensional alcoved simplex in $\T_P$ and hence it appears in the first part of the representation of $c_{d-1}^b(P)$.

More precisely, if the facet is contained in the interior $\acirc{P}$ of~$P$, then it is a facet of exactly two $d$-dimensional alcoved simplices in~$\T_P$, and so it doesn't contribute at all to $c_{d-1}^b(P)$.
If the facet $F$ is not contained in the interior, then it is a facet of exactly one alcoved simplex and it contributes $\frac12 (b-1)^{d-1} \rvol(F)$ to~$c_{d-1}^b(P)$.
\end{proof}

Based on this representation we can now prove that the tropical barycentric $(d-1)$-volumes bound the second highest tropical Ehrhart coefficient.

\begin{theorem}
\label{thm:Logcd-1-bound}
If $P \subseteq \TT^d$ is a tropical lattice polytope, then
\[
\tlvol_{d-1}^-(P) \leq \Log{c^b_{d-1}(P)} \leq \tlvol_{d-1}^+(P).
\]
\end{theorem}

\begin{proof}
Our arguments are based on the representation of $c_{d-1}^b(P)$ given in Lemma~\ref{lem:second+highest+trop+coeff}.
We start with the claimed lower bound.
As a minimum of linear functions, the function
\[
x \mapsto \min\left\{ v^\intercal x : v \in \{0,1\}^d, \one^\intercal v = d-1 \right\}
\]
attains its maximum over $\Tr_{d-1}(P)$ at a boundary point and thus on a $(d-1)$-dimensional alcoved simplex $\Delta_\pi^s(a) \in \T_P$ that has a non-zero contribution to $c_{d-1}^b(P)$.
Since the boundary of the $(d-1)$-trunk of~$P$ is triangulated by the closures of those $\Delta_\pi^s(a)$, it suffices to show that for these simplices
\begin{align}
\Log{\rvol(D_b^a\Delta^s_\pi(\zero))} + d-1 &\geq \max_{x \in \Delta_\pi^s(a)} \min\Big\{ v^\intercal x : v \in \{0,1\}^d, \one^\intercal v = d-1 \Big\}.\label{eqn:valrvol}
\end{align}
First of all, by symmetry we only need to consider $\pi=id$.
In order to compute the relative volume of $D_b^a\Delta^s(\zero)$, we note that there are indices $0 \leq j_0 < j_1 < \ldots < j_{d-1} \leq d$ such that the closure of $\Delta^s(\zero)$ is given by
\[
\overline{\Delta^s(\zero)} = \conv\left\{e_{[j_0]},e_{[j_1]},\ldots,e_{[j_{d-1}]}\right\},
\]
where $e_{[j]} = e_1 + \ldots + e_j$ and $e_{[0]} = \zero$.
The linear subspace parallel to the affine span of $D_b^a \Delta^s(\zero)$ is thus given by
\begin{align*}
L_b^a(s) &= \lin\left\{ D_b^a(e_{[j_1]} - e_{[j_0]}),D_b^a(e_{[j_2]} - e_{[j_0]}),\ldots,D_b^a(e_{[j_{d-1}]} - e_{[j_0]}) \right\}\\
&= \lin\left\{ D_b^a(e_{[j_1]} - e_{[j_0]}),D_b^a(e_{[j_2]} - e_{[j_1]}),\ldots,D_b^a(e_{[j_{d-1}]} - e_{[j_{d-2}]}) \right\}.
\end{align*}
The determinant of the $(d-1)$-dimensional sublattice $\Z^d \cap \aff(D_b^a\Delta^s(\zero))$ can be estimated by
\begin{align}
\det\left(\Z^d \cap \aff(D_b^a\Delta^s(\zero))\right) = \det\left(\Z^d \cap L_b^a(s)\right) \leq \det\left(V^\intercal V\right)^{\frac12},\label{eqn:sublatticebound}
\end{align}
where $V \in \Z^{d \times (d-1)}$ is any matrix whose columns $\{v_1,\ldots,v_{d-1}\}$ are linearly independent vectors from $\Z^d \cap L_b^a(s)$.
Note that
\[
D_b^a(e_{[j_l]} - e_{[j_{l-1}]}) = b^{a_{j_{l-1}+1}} e_{j_{l-1}+1} + \ldots + b^{a_{j_l}} e_{j_l} =: \overline{v_l},
\]
so that $v_l := b^{-\min\{a_{j_{l-1}+1},\ldots,a_{j_l}\}} \cdot \overline{v_l} \in \Z^d \cap L_b^a(s)$, for every $l=1,\ldots,d-1$.
Here we used that $P$ is a tropical lattice polytope, and thus $a \in \Z^d_{\geq0}$.

Before we use~\eqref{eqn:sublatticebound} to estimate the determinant of said sublattice, we observe that $(d-1)!\cdot\vol_{d-1}(D_b^a\Delta^s(\zero))$ equals the $(d-1)$-volume of the parallelepiped spanned by $D_b^a(e_{[j_1]} - e_{[j_0]}),D_b^a(e_{[j_2]} - e_{[j_0]}),\ldots,D_b^a(e_{[j_{d-1}]} - e_{[j_0]})$, which in turn equals the $(d-1)$-volume of the parallelepiped~$Q_{d-1}$ spanned by $\overline{v_1},\ldots,\overline{v_{d-1}}$.
We have, $\overline{v}_l^\intercal \overline{v}_k = 0$, for $l \neq k$, and $\overline{v}_l^\intercal \overline{v}_l = \sum_{r=j_{l-1}+1}^{j_l} b^{2a_r}$.
Hence, $\overline{V}^\intercal \overline{V}$ is a diagonal matrix and evaluating its determinant gives the following formula that we record for later use
\begin{align}
\vol_{d-1}(D_b^a \Delta^s(\zero)) &= \frac{1}{(d-1)!} \prod_{t=0}^{d-2} \left( \sum_{\ell=j_t+1}^{j_{t+1}} b^{2 a_\ell} \right)^{\frac12}.\label{eqn:voliDbaDelta}
\end{align}
Now, using~\eqref{eqn:sublatticebound} for the matrix $V=(v_1,\ldots,v_{d-1})$, we get
\begin{align*}
\det\left(\Z^d \cap L_b^a(s)\right) &\leq \prod_{l=1}^{d-1} b^{-\min\{a_{j_{l-1}+1},\ldots,a_{j_l}\}} \det(\overline{V}^\intercal \overline{V})^{\frac12}\\
&= b^{-\sum_{l=1}^{d-1} \min\{a_{j_{l-1}+1},\ldots,a_{j_l}\}} \vol_{d-1}(Q_{d-1}).
\end{align*}
Putting things together we arrive at the following lower bound on the relative volume of $D_b^a\Delta^s(\zero)$:
\[
\rvol(D_b^a\Delta^s(\zero)) = \frac{\vol_{d-1}(D_b^a\Delta^s(\zero))}{\det(\Z^d \cap \aff(D_b^a\Delta^s(\zero)))} \geq \tfrac{1}{(d-1)!} \cdot b^{\sum_{l=1}^{d-1} \min\{a_{j_{l-1}+1},\ldots,a_{j_l}\}}.
\]
Now, the map $\Log{\cdot}$ is monotone in the sense that $\Log{f} \geq \Log{g}$ whenever $\card{f(b)} \geq \card{g(b)}$ for all $b \in \N$.
Therefore,
\begin{align*}
\Log{\rvol(D_b^a\Delta^s(\zero))} &\geq \Log{\tfrac{1}{(d-1)!} \cdot b^{\sum_{l=1}^{d-1} \min\{a_{j_{l-1}+1},\ldots,a_{j_l}\}}}\\
&= \sum_{l=1}^{d-1} \min\{a_{j_{l-1}+1},\ldots,a_{j_l}\},
\end{align*}
and so for~\eqref{eqn:valrvol} it suffices to show that
\begin{align}
&\phantom{=} \sum_{l=1}^{d-1} \min\{a_{j_{l-1}+1},\ldots,a_{j_l}\} + d-1 \label{eqn:valrvol2nd}\\
&\geq \max_{x \in \Delta^s(a)} \min\left\{ v^\intercal x : v \in \{0,1\}^d, \one^\intercal v = d-1 \right\}.\nonumber
\end{align}
The maximum on the right hand side is attained at a vertex of $\Delta^s(a)$, that is, at a point of the form $a + e_{[j_l]}$, $l=1,\ldots,d-1$.
It thus evaluates to
\begin{align*}
&\phantom{=} \max_{l=1,\ldots,d-1} \,\min\left\{ v^\intercal (a+e_{[j_l]}) : v \in \{0,1\}^d, \one^\intercal v = d-1 \right\}\\
&= \min\left\{ v^\intercal (a+e_{[j_{d-1}]}) : v \in \{0,1\}^d, \one^\intercal v = d-1 \right\}.
\end{align*}
Since $j_0 < \ldots < j_{d-1}$, this implies~\eqref{eqn:valrvol2nd} and thus the claimed lower bound on $\Log{c_{d-1}^b(P)}$.

\smallskip

We now prove the upper bound.
%
First note that the determinant of a $(d-1)$-dimensional sublattice~$L$ of $\Z^d$ is at least~$1$.
Indeed, there always exists a non-zero vector $u \in \Z^d$ such that $\det(L) = \|u\| \geq 1$ (cf.~\cite[Cor.~1.3.5]{martinet2003perfect}).

Now, let us consider an alcoved simplex $\Delta_\pi^s(a) \in \T_P$ with $\dim(\Delta_\pi^s(a))=d-1$.
Again by symmetry, we can concentrate on $\pi=id$.
As before, we find indices $0 \leq j_0 < j_1 < \ldots < j_{d-1} \leq d$ such that
\[
\overline{\Delta^s(\zero)} = \conv\left\{e_{[j_0]},e_{[j_1]},\ldots,e_{[j_{d-1}]}\right\}.
\]
The identity~\eqref{eqn:voliDbaDelta} yields
\begin{align*}
\rvol(D_b^a\Delta^s(\zero)) &= \frac{\vol_{d-1}(D_b^a\Delta^s(\zero))}{\det(\Z^d \cap \aff(D_b^a\Delta^s(\zero)))} \leq \vol_{d-1}(D_b^a\Delta^s(\zero))\\
&= \frac{1}{(d-1)!} \prod_{t=0}^{d-2} \left( \sum_{\ell=j_t+1}^{j_{t+1}} b^{2 a_\ell} \right)^{\frac12}.
\end{align*}
Therefore,
\begin{align*}
&\ \Log{(b-1)^{d-1} \rvol(D_b^a\Delta_\pi^s(\zero))} \leq d-1 + \sum_{t=0}^{d-2} \max\{a_{j_t+1},\ldots,a_{j_{t+1}}\}\\
&\leq \max_{x \in \Delta_\pi^s(a)} \max\left\{v^\intercal x : v \in \{0,1\}^d, \one^\intercal v = d-1\right\} \leq \tlvol_{d-1}^+(P).
\end{align*}
%
%
Since $\Log{f+g} \leq \max\{\Log{f},\Log{g}\}$, the formula in Lemma~\ref{lem:second+highest+trop+coeff} gives us
\begin{align*}
\Log{c_{d-1}^b(P)} &\leq \max_{\substack{\Delta_\pi^s(a) \in \T_P\\ \dim(\Delta_\pi^s(a)) = d-1}} \Log{\delta(\Delta_\pi^s(a)) \cdot (b-1)^{d-1} \rvol(D_b^a\Delta^s_\pi(\zero))} \\
&\leq \tlvol_{d-1}^+(P),
\end{align*}
finishing the proof.
\end{proof}

\begin{remark} \quad \label{rem:behaviour+lower+dim+vol}
\begin{enumerate}[i)]
 \item \label{discrepancy+lower+upper} The inequality in Theorem~\ref{thm:Logcd-1-bound} can be strict.
Indeed, consider the matrix $M = \left( \begin{array}{ccc} \ell-1 & \ell & k+\ell \\ 0 & 0 & k+1 \end{array} \right)$ and observe that by the computations in Example~\ref{ex:triangle+edge}, the tropical lattice polygon $P = \tconv(M)$ has parameters 
\[
\left(\tlvol_1^-(P),\Log{c_1^b(P)},\tlvol_1^+(P)\right) = \Big(k+1,\max\{\ell,k+1\},k+\ell\Big) \enspace .
\]

 \item The assumption that the generating matrix $M$ of~$P$ has only non-negative entries cannot be dropped in general.
 For instance, consider the tropical triangle~$\Delta_2$ with defining matrix $M = \left( \begin{array}{ccc} 1 & 0 & -1 \\ 1 & -1 & 0 \end{array} \right)$.
 We find that
 \[
\tlvol_1^-(\Delta_2) = 1 \quad \text{and} \quad \Log{c_1^b(\Delta_2)} = 0.
\]
 \item It suffices however to make the weaker assumption $\Tr_{d-1}(P) \subseteq \R^d_{\geq0}$.
 The argument is based on the quasi-polynomiality of the counting function $\TL_P^b(k)$, for tropical polytopes~$P$ with a not necessarily non-negative integral defining matrix (cf.~Remark~\ref{rem:quasi-polynomials}).
\end{enumerate}
\end{remark}

\subsection{Tropical \texorpdfstring{$i$}{i}-minors}

In this part, we aim to extend Theorem~\ref{thm:tvol_qtvol} in order to give an upper estimate for the tropical lower barycentric $i$-volume in terms of tropical analogs of $i$-minors of the defining matrix~$M$ of~$P$.

\begin{definition}[Maximal tropical $i$-minor]
Let $M \in \TT^{d \times m}$ be a tropical matrix and let $i \in \{1,2,\ldots,\min\{d,m\}\}$.
We define the \emph{maximal tropical $i$-minor} of~$M$ as
\[
\tminor_i(M) := \max_{I \in \binom{[d]}{i}, J \in \binom{[m]}{i}} \tdet(M_{I,J}),
\]
where $M_{I,J}$ is the $i\times i$ submatrix of $M$ whose rows are indexed by $I$ and whose columns are indexed by $J$.
\end{definition}

We need a generalization of~\cite[Prop.~15]{DepersinGaubertJoswig:2017} to all maximal tropical $i$-minors.
In order to state it, we record that they call a matrix $M \in \TT^{d \times m}$ \emph{tropically sign-generic} if for each $J \in \binom{[m]}{d}$ all permutations attaining $\tdet(M_J)$ have the same sign.

\begin{lemma}
\label{lem:tmi-monotone}
Let $A \in \TT^{d \times m}$ and $B \in \TT^{d \times n}$ be such that $\tconv(A) \subseteq \tconv(B)$.
If there are $I \in \binom{[d]}{i}$ and $J \in \binom{[m]}{i}$ such that $A_{I,J}$ is tropically sign-generic and $\tminor_i(A) = \tdet(A_{I,J})$, then $\tminor_i(A) \leq \tminor_i(B)$.
\end{lemma}

\begin{proof}
We follow the arguments in the proof of~\cite[Prop.~15]{DepersinGaubertJoswig:2017}.
First of all, since $\tconv(A) \subseteq \tconv(B)$, there exists a matrix $C \in \TT^{n \times m}$ with non-positive entries such that $A = B \odot C$.
Indeed, this is a compact way to write that every column of $A$ is a tropical convex combination of the columns of~$B$.

Let $I \in \binom{[d]}{i}$ and $J \in \binom{[m]}{i}$ have the assumed properties.
Then, by the tropical analog of the Cauchy-Binet formula (cf.~\cite[Thm.~5.4]{PoplinHartwig:2004} or~\cite[Ex.~3.7]{AkianGaubertGuterman:2009}), we have
\begin{align}
&\ \card{A_{I,J}}^+ \oplus \max_{K \in \binom{[n]}{i}} \max\left\{\card{B_{I,K}}^+ + \card{C_{K,J}}^-,\card{B_{I,K}}^- + \card{C_{K,J}}^+\right\}\nonumber\\
&= \card{A_{I,J}}^- \oplus \max_{K \in \binom{[n]}{i}} \max\left\{\card{B_{I,K}}^+ + \card{C_{K,J}}^+,\card{B_{I,K}}^- + \card{C_{K,J}}^-\right\},\label{eqn:tropCauchyBinet}
\end{align}
where the pair
\[
\card{M}^+ = \max_{\substack{\sigma \in S_i\\ \sgn(\sigma)=1}}\sum_{l=1}^i m_{l\sigma(l)} \quad \text{and} \quad \card{M}^- = \max_{\substack{\sigma \in S_i\\\sgn(\sigma)=-1}}\sum_{l=1}^i m_{l\sigma(l)},
\]
forms the \emph{bideterminant} $(\card{M}^+,\card{M}^-)$ of $M \in \TT^{i \times i}$.
We write $\card{M}^\pm$ in order to refer to either of the two components of the bideterminant.

Now, since $A_{I,J}$ is tropically sign-generic, we have $\card{A_{I,J}}^+ \neq \card{A_{I,J}}^-$, and thus by~\eqref{eqn:tropCauchyBinet}
\[
\tdet(A_{I,J}) \leq \max_{K \in \binom{[n]}{i}} \max\left\{\card{B_{I,K}}^\pm + \card{C_{K,J}}^\pm\right\} \leq \max_{K \in \binom{[n]}{i}} \tdet(B_{I,K}),
\]
where we also used that the entries of~$C$ are non-positive.
Therefore,
\[
\tminor_i(A) = \tdet(A_{I,J}) \leq \max_{I \in \binom{[d]}{i}, K \in \binom{[n]}{i}}\tdet(B_{I,K}) = \tminor_i(B),
\]
as desired.
\end{proof}

\begin{lemma}
\label{lem:trop-sign-gener-simplices}
Let $\pi \in S_d$ be a permutation, let $0 \leq j_0 < j_1 < \ldots < j_i \leq d$ be indices, and let $S \in \TT^{d \times (i+1)}$ be the matrix with columns $S_l = e^\pi_{[j_l]}$, for $l=0,1,\ldots,i$.
Then, there are $I \in \binom{[d]}{i}$ and $J \in \binom{[i+1]}{i}$ such that $\tminor_i(S) = \tdet(S_{I,J})$ and $S_{I,J}$ is tropically sign-generic.
\end{lemma}

\begin{proof}
First of all, the statement and in particular $\tminor_i(S)$ is invariant under permutations of the rows of $S$.
Thus, we may assume that $\pi=id$.
Second, $\tminor_i(S) = i$ and it is attained by the $i \times i$-matrix arising from~$S$ after deleting the first column and keeping the rows corresponding to $j_1,\ldots,j_i$.
More precisely, $\tminor_i(S) = \tdet(S_{I,J})$ for $I=\{j_1,\ldots,j_i\}$ and $J=[i+1] \setminus \{1\}$.
Furthermore, $S_{I,J}$ is an upper triangular matrix with $1$'s on the diagonal.
Thus, $\tdet(S_{I,J})$ is uniquely attained by the identity permutation and so $S_{I,J}$ is tropically sign-generic.
\end{proof}

In view of~\eqref{eqn:DGJ} and the identity $\tlvol(P) = \Log{c^b_d(P)}$, the following extends Theorem~\ref{thm:tvol_qtvol} to all tropical lower barycentric $i$-volumes.

\begin{theorem}
\label{thm:tlvoli_tmi}
Let $M \in \TT\N^{d \times m}$ and let $P = \tconv(M)$ be the corresponding tropical lattice polytope.
Then, for every $i \in [d]$, we have
\[
\tlvol_i^-(P) \leq \tminor_i(M).
\]
\end{theorem}

\begin{proof}
The $i$-trunk of $P$ is the union of all $(\geq i)$-dimensional alcoved simplices occurring in the covector decomposition of~$P$.
If $\Tr_i(P) = \emptyset$, then $\tlvol_i^-(P) = -\infty$ and there is nothing to prove.
We thus assume otherwise, and we let $\Delta_\pi^s(a) \subseteq \Tr_i(P)$ be an alcoved simplex with $\dim(\Delta_\pi^s(a)) \geq i$.
Of course, it suffices to show that
\begin{align}
\max_{x \in \Delta_\pi^s(a)} \,\min\left\{ v^\intercal x : v \in \{0,1\}^d, \one^\intercal v = i \right\} &\leq \tminor_i(M).\label{eqn:tlvoli-vs-tmi-reduction}
\end{align}
Observe that the maximum on the left hand side is attained at a boundary point of $\Delta_\pi^s(a)$, so that we can assume without loss of generality that $\dim(\Delta_\pi^s(a)) = i$.
There are indices $0 \leq j_0 < j_1 < \ldots < j_i \leq d$ such that
\[
\Delta_\pi^s(a) = a + \conv\left\{e^\pi_{[j_0]},e^\pi_{[j_1]},\ldots,e^\pi_{[j_i]}\right\},
\]
where $e^\pi_{[l]} = e_{\pi(1)} + \ldots + e_{\pi(l)}$.
Let $S \in \TT^{d \times (i+1)}$ be the matrix whose columns correspond to the $i+1$ vertices of $\Delta_\pi^s(a)$.

Combining Lemma~\ref{lem:tmi-monotone}, Lemma~\ref{lem:trop-sign-gener-simplices}, and $\tconv(S) = \Delta_\pi^s(a) \subseteq P = \tconv(M)$, we see that $\tminor_i(S) \leq \tminor_i(M)$.
Thus, for~\eqref{eqn:tlvoli-vs-tmi-reduction} it suffices to show
\begin{align*}
\max_{x \in \Delta_\pi^s(a)} \,\min\left\{ v^\intercal x : v \in \{0,1\}^d, \one^\intercal v = i \right\} &\leq \tminor_i(S).
\end{align*}
To this end, we first observe that by symmetry we may assume that $\pi=id$ and that $j_0=0$.
Moreover, the maximum on the left hand side is attained at the point $\widebar a = a + e_{[j_i]}$, since every $x \in \Delta^s(a)$ is coordinate-wise dominated by $\widebar a$ and because the function $x \mapsto v^\intercal x$ is non-decreasing with respect to this partial order.

Now, the $r^{th}$ coordinate of $\widebar a$ is given by $\widebar a_r = a_r + 1$, if $r \leq j_i$, and $\widebar a_r = a_r$, if $r > j_i$.
Therefore, $\widebar a_{j_l} = S_{j_l,l}$ for every $1 \leq l \leq i$.
For $\widebar v \in \{0,1\}^d$ defined by $\widebar v_r = 1$ if and only if $r \in \{j_1,\ldots,j_i\}$, we thus obtain
\[
\max_{x \in \Delta^s(a)} \,\min\left\{ v^\intercal x : v \in \{0,1\}^d, \one^\intercal v = i \right\} \leq \widebar v^\intercal \widebar a = \sum_{l=1}^i S_{j_l,l} \leq \tminor_i(S).\qedhere
\]
\end{proof}

\noindent For $i=1$ there is a more direct argument that gives a stronger result and allows to drop the integrality assumption:

\begin{proposition}
\label{prop:tlvol1-tm1}
Let $M \in \TT^{d \times m}$ and let $P = \tconv(M)$ be the corresponding tropical polytope.
Then
\[
\tlvol_1^-(P) \leq \tlvol_1^+(P) = \tminor_1(M),
\]
and equality holds if and only if $\tminor_1(M)\cdot\one$ is the tropical barycenter of~$P$.
\end{proposition}

\begin{proof}
First of all, $\tminor_1(M) = \max_{1 \leq i \leq d, 1 \leq j \leq m} M_{i,j}$ is just the maximal entry of~$M$.
Moreover, for every $x \in P=\tconv(M)$ there are coefficients $\gamma_1,\ldots,\gamma_m \in \TT$ with $\bigoplus_{j=1}^m \gamma_j = 0$ and $x = \bigoplus_{j=1}^m \gamma_j \odot M_{\cdot,j}$.
Since also $P=\Tr_1(P)$, we have
\begin{align*}
\tlvol_1^-(P) &= \max_{x \in P} \min_{1 \leq i \leq d} x_i\\
&= \max_{\gamma_1 \oplus \ldots \oplus \gamma_m = 0} \min_{1 \leq i \leq d} \max\{\gamma_1 + M_{i,1},\ldots,\gamma_m + M_{i,m}\}\\
&\leq \max_{\gamma_1 \oplus \ldots \oplus \gamma_m = 0} \max_{1 \leq i \leq d} \max\{\gamma_1 + M_{i,1},\ldots,\gamma_m + M_{i,m}\} = \tminor_1(M).
\end{align*}
Equality holds if and only if there exist coefficients $\gamma_1,\ldots,\gamma_m \in \TT$ with $\gamma_1 \oplus \ldots \oplus \gamma_m = 0$ such that
\[
\min_{1 \leq i \leq d} \max\{\gamma_1 + M_{i,1},\ldots,\gamma_m + M_{i,m}\} = \max_{1 \leq i \leq d, 1 \leq j \leq m} M_{i,j}.
\]
This happens if and only if each row $M_{i,\cdot}$ contains a maximal entry of~$M$.
The corresponding coefficients would just be $\gamma_1 = \ldots = \gamma_m = 0$.
In other words, the tropical barycenter of $P$ equals $\tminor_1(M)\cdot\one$.
\end{proof}

We conjecture that the maximal tropical $i$-minors also upper bound the corresponding tropical Ehrhart coefficients, and that the following analogous bound to Theorem~\ref{thm:tlvoli_tmi} holds:

\begin{conjecture}
Let $M \in \TT\N^{d \times m}$ and let $P = \tconv(M)$ be the corresponding tropical lattice polytope.
Then, for $i \in [d]$, we have
\[
\Log{c^b_i(P)} \leq \tminor_i(M).
\]
\end{conjecture}

\begin{example}
For $\ell \geq 2$, consider the example $M = \left( \begin{array}{ccc} 0 & 0 & \ell-1 \\ 0 & 1 & \ell-1 \end{array} \right)$ again (see also Fig.~\ref{fig:latpoly+3}). 
Writing $P = \tconv(M)$, we have
\begin{align*}
\TL_P^b(k) &= \tfrac12(b-1)^2 (b^k)^2 + \tfrac12 ( b^{\ell-1} + 2b - 3 ) (b^k) + 1.
\end{align*}
Thus, $\Log{c^b_2(P)} = 2 \leq \ell = \tminor_2(M)$ and $\Log{c^b_1(P)} = \ell - 1 = \tminor_1(M)$.
\end{example}

\subsection{Tropical surface areas}

We end this section with a few musings on reasonable surface area concepts for tropical polytopes that naturally evolve from our previous studies.
For one, the tropical barycentric $(d-1)$-volumes may serve as surface areas.
Let us thus define the \emph{upper} and \emph{lower tropical surface area} of a tropical polytope $P \subseteq \TT^d$ as
\[
\tlsurf^+(P) := \tlvol_{d-1}^+(P) \qquad \text{and} \qquad \tlsurf^-(P) := \tlvol_{d-1}^-(P),
\]
respectively.

On the other hand, the second highest Ehrhart coefficient of an ordinary lattice polytope $Q \subseteq \R^d$ is a kind of \emph{discrete} surface area (cf.~\cite[Thm.~5.6]{beckrobins2007computing}).
More precisely, writing $\shp{k Q \cap \Z^d} = \sum_{i=0}^d c_i(Q) k^i$, we have
\begin{align}
c_{d-1}(Q) = \frac12 \sum_{F\text{ a facet of }Q} \frac{\vol_{d-1}(F)}{\det(\Z^d \cap \aff(F))} = \frac12 \sum_{F\text{ a facet of }Q} \rvol(F).\label{eqn:classical-c-d-1}
\end{align}
In this spirit, we may call
\[
\Log{c_{d-1}^b(P)}
\]
the \emph{discrete tropical surface area} of a tropical lattice polytope $P \subseteq \TT^d$.
Also, the formula for $c_{d-1}^b(P)$ in Lemma~\ref{lem:second+highest+trop+coeff} suggests this as a surface area concept.

Natural questions for future studies arise from these definitions.
First of all, we may ask for an isoperimetric inequality for tropical polytopes.
The precise question taking the homogeneity of the magnitudes into account is as follows:

\begin{question}
Are there constants $c_d^+,c_d^- \in \TT$ only depending on the dimension~$d$, such that
\[
\tlvol(P)^{\odot (d-1)} \leq c_d^\pm \odot \tlsurf^\pm(P)^{\odot d},
\]
for every tropical polytope $P \subseteq \TT^d$?
\end{question}

Depersin, Gaubert \& Joswig~\cite{DepersinGaubertJoswig:2017} established an isodiametric inequality for tropical simplices with respect to the functional $\tvol(\cdot)$ discussed in Section~\ref{sssect:tvol}, and obtained interesting families of tropical polytopes along the way.
We thus ask

\begin{question}
Is there an interesting isodiametric inequality with respect to $\tlvol(\cdot)$?
\end{question}

Regarding discrete surface area measures, we remark that Bey, Henk \& Wills~\cite[Prop.~4.2]{beyhenkwills2007notes} proved an isoperimetric type inequality for lattice polytopes $Q \subseteq \R^d$.
It states that $c_{d-1}(Q) \leq \binom{d+1}{2} \vol(Q)$.

\begin{question}
Does there exist a discrete isoperimetric inequality relating $\tlvol(\cdot)$ and $\Log{c_{d-1}^b(P)}$?
\end{question}

\section{Computational aspects}
\label{sec:computational+aspects}

A matrix $M \in \TT^{r \times r}$ is called \emph{non-singular} if the value of the tropical determinant is attained at most once. 
The \emph{tropical rank} $\trk(M)$ of a matrix $M \in \TT^{d \times m}$ is the size of a largest non-singular square submatrix of~$M$.
This notion was introduced and studied by Develin, Santos \& Sturmfels who prove in~\cite[Thm.~4.2]{DevelinSantosSturmfels:2005} that the tropical rank equals the dimension of~$P=\tconv(M)$ seen as a polytopal complex.

Recall that by Theorem~\ref{thm:tropEhrCoeffsRep} there exists an $i$-dimensional element in~$\F_P$, if the tropical Ehrhart coefficient $c_i^b(P)$ is non-vanishing.
This readily implies


\begin{lemma}
\label{lem:trk-equals-nonnullEC}
Let $M \in \TT\N^{d \times m}$ and let $P = \tconv(M)$. Then,
\[
\trk(M) \geq \max\left\{i : c_i^b(P) \neq 0 \right\}.
\]
\end{lemma}

Kim \& Roush~\cite[Thm.~13]{kimroush2005factorizations} showed that deciding if $\trk(M) \geq k$ is NP-complete.
Their proof shows that this is true even for $0/1$-matrices and thus we conclude

\begin{theorem}
  Let $P \subseteq \TT^d$ be a tropical lattice polytope.
  Deciding whether $\max\left\{i : c_i^b(P) \neq 0 \right\} \geq k$ is in general NP-hard.
\end{theorem}

Deciding whether the tropical barycentric volume $\tlvol(P)=\Log{c_d^b(P)}$ is non-vanishing is a supposedly easier problem.
For example, if $P$ is a pure tropical lattice polytope, then by Corollary~\ref{cor:pure}, we have $\tlvol(P) = \qtvol^+(M)$.
In this case, the latter magnitude and thus $\tlvol(P)$ can be computed in time $O(m^3)$ as shown in~\cite{DepersinGaubertJoswig:2017}.
On the other hand, this decision problem is equivalent to checking non-singularity of the defining matrix~$M$, which is equivalent to checking feasibility of a tropical linear program or deciding winning positions in mean-payoff games and lies in NP~$\cap$~coNP (cf.~\cite[\S\,2.2]{gaubertmaccaig2019approximating}).

\begin{proposition}
  Computing the tropical barycentric volume $\tlvol(P)$ is at least as hard as checking feasibility of a tropical linear inequality system. 
\end{proposition}

One way to compute the tropical barycentric volume in Definition~\ref{def:tropical+volume+from+lattice} is via the explicit determination of the covector decomposition, see~\cite{Joswig:2009}, involving a classical convex hull computation.

We propose another possibility which is closer to the computation of the tropical dequantized volume defined in Definition~\ref{def:dequantized+volume}.
For this, we start by considering a \emph{tropical simplex}, namely the tropical convex hull of a $d \times (d+1)$ matrix $A \in \TT^{d \times (d+1)}$.
We let $\widebar{A} \in \TT^{(d+1) \times (d+1)}$ arise from~$A$ by appending a zero-th row filled with tropical ones $0$.
With the Hungarian method, one can compute the permutation attaining the tropical determinant $\tdet(\widebar{A})$ in $O(d^3)$, see~\cite[\S\,1.6.4]{Butkovic:2010}. 
Using the dual variables and reordering the columns, we can assume that the tropical determinant $\tdet(\widebar{A})$ is attained at the identity permutation, that all entries on the diagonal are~$0$ and that all off-diagonal entries are non-positive.
One can deduce from~\cite[Lem.~4.3.2]{Butkovic:2010} that the columns of the \emph{Kleene star} $\widebar{A}^* = \widebar{A} \oplus \widebar{A}^2 \oplus \dots \oplus \widebar{A}^d$ provide generators of the $d$-trunk of $\tconv(A)$, by appropriately scaling so that the zero-th row consists only of $0$'s again.
Computing the Kleene star takes again $O(d^3)$ time.
In summary, we have

\begin{proposition} \label{prop:barycenter+dtrunk+trop+simplex}
  The $d$-trunk of a tropical simplex in $\TT^d$ is a polytrope.
  Its tropical barycenter can be computed in time $O(d^3)$.
\end{proposition}

We denote the tropical barycenter of the $d$-trunk of a tropical simplex $S \subseteq \TT^d$ by~$\bt(S)$.

\begin{proposition}
Let $M \in \TT^{d \times m}$ and let $P = \tconv(M)$.
The tropical barycentric volume $\tlvol(P)$ is the maximum
  \[
  \tlvol(P) = \max_{J \in \binom{[m]}{d+1}} \bigg\{ \sum_{i=1}^d \bt(\tconv(M_J))_i : \widebar{M_J}~\text{non-singular} \bigg\} \enspace .
  \]
  It can be computed in time $O(\binom{m}{d+1} \cdot d^3)$.
\end{proposition}
\begin{proof}
  By the tropical Carath\'{e}odory theorem, the tropical convex hull of~$M$ is the union of the tropical simplices $\tconv(M_J)$, $J \in \binom{[m]}{d+1}$.
  We compute the tropical barycenter of each of these tropical simplices in $O(d^3)$ time by Proposition~\ref{prop:barycenter+dtrunk+trop+simplex}. 
\end{proof}

\begin{remark}
  One could consider the tropical barycentric volume $\tlvol(P)$ as a robust version of a transportation problem.
  The tropical dequantized volume is the generalization of a maximal matching problem, namely a transportation problem~\cite[Cor.~18]{DepersinGaubertJoswig:2017}.
  The tropical barycentric volume is the solution of the transportation problem for its $d$-trunk, without the lower-dimensional parts.
  In this sense, it is more robust with respect to perturbations.
\end{remark}

\begin{question}
Let $P \subseteq \TT^d$ be a tropical polytope.
\begin{enumerate}[i)]
 \item How fast can we compute $\tlvol(P)$?
 \item What is the computational complexity of deciding $\tlvol(P) \neq -\infty$?
\end{enumerate}
\end{question}

\noindent Note that computing the volume of an ordinary polytope is \#P-hard (\cite{DyerFrieze:1988}).

\section{Acknowledgments}

We thank Josephine Yu for motivating us with her interest in tropical Ehrhart theory.
Further we thank Michael Joswig and Christoph Hunkenschr\"oder for helpful conversations, and Raman Sanyal for pointing us to the paper of Bj\"orner \& Wachs~\cite{bjoernerwachs1996shellableI} and for comments on an earlier draft of the paper.
Thanks to Leon Zhang for communicating related work.

\bibliographystyle{amsplain}
\bibliography{mybib}

\end{document}